\documentclass[onefignum,onetabnum]{siamart250211}

\newtheorem{assumption}{Assumption}[section] 
\newtheorem{example}{Example}
\newtheorem{remark}[theorem]{Remark}

\def\cC{{\mathcal{C}}}

\def\cN{{\mathcal{N}}}

\def\cP{{\mathcal{P}}}

\def\cX{\mathcal{X}}
\def\cY{\mathcal{Y}}

\newcommand{\dom}{\operatorname{dom}}

\newcommand{\R}{\mathbb{R}}

\DeclareMathOperator*{\argmin}{arg\,min}

\newcommand{\irchi}[2]{\raisebox{\depth}{$#1\chi$}} 
\usepackage[normalem]{ulem}

\newcommand{\beps}{\boldsymbol{\varepsilon}}

\newcommand{\bid}{\boldsymbol{\mathrm{Id}}}
\newcommand{\bA}{\mathbf{A}}
\newcommand{\ba}{\mathbf{a}}
\newcommand{\bx}{\mathbf{x}}
\newcommand{\bu}{\mathbf{u}}
\newcommand{\by}{\mathbf{y}}
\newcommand{\bz}{\mathbf{z}}

\usepackage{lipsum}
\usepackage{graphicx}
\usepackage{stackengine}
\usepackage{amsfonts}
\usepackage{amssymb}
\usepackage{graphicx}
\usepackage{epstopdf}
\usepackage{algorithmic}
\usepackage{comment}
\usepackage{dsfont}
\usepackage{hyperref}
\usepackage{enumitem}

\ifpdf
  \DeclareGraphicsExtensions{.eps,.pdf,.png,.jpg}
\else
  \DeclareGraphicsExtensions{.eps}
\fi

\newsiamremark{hypothesis}{Hypothesis}
\crefname{hypothesis}{Hypothesis}{Hypotheses}
\newsiamthm{claim}{Claim}
\newsiamthm{prop}{Proposition}
\newsiamthm{defi}{Definition}

\headers{Optimization landscape of $\ell_0$-Bregman relaxations}{J. Chirinos-Rodriguez, C. F\'evotte, E. Soubies}

\title{Optimization landscape of $\ell_0$-Bregman relaxations{\renewcommand{\thefootnote}{}\thanks{\today}}}

\author{Jonathan Chirinos-Rodríguez\textsuperscript{*,1} \and C\'edric F\'evotte\textsuperscript{1} \and Emmanuel Soubies\textsuperscript{1}}

\usepackage{amsopn}

\graphicspath{{./Images}}

\ifpdf
\hypersetup{
  pdftitle={Optimization landscape of l0-Bregman relazations},
  pdfauthor={J. Chirinos-Rodríguez, C. Févotte, E. Soubies}
}
\fi

\begin{document}
\sloppy

\renewcommand{\thefootnote}{\fnsymbol{footnote}}
\footnotetext[1]{Corresponding author.} 

\renewcommand{\thefootnote}{\arabic{footnote}}
\footnotetext[1]{IRIT, CNRS, Toulouse INP, Université de Toulouse, France (\email{jonathan-eduardo.chirinos-rodriguez@irit.fr}, \email{cedric.fevotte@irit.fr}, \email{emmanuel.soubies@irit.fr})}

\maketitle

\begin{abstract}
In this paper, we study (noisy) linear systems, and their $\ell_0$-regularized optimization problems, coupled with general data fidelity terms. Recent approaches for solving this class of problems have proposed to consider non-convex exact continuous relaxations that preserve global minimizers while reducing the number of local minimizers. Within this framework, we consider the class of $\ell_0$-Bregman relaxations, and establish sufficient conditions under which a critical point is isolated in terms of sparsity, in the sense that any other critical point has a strictly larger cardinality. In this way, we ensure a form of uniqueness in the solution structure. Furthermore, we analyze the exact recovery properties of such exact relaxations. To that end, we derive conditions under which the oracle solution (i.e., the one sharing the same support as the ground-truth) is the unique global minimizer of the relaxed problem, and is isolated in terms of sparsity. Our analysis is primarily built upon a novel property we introduce, termed the Bregman Restricted Strong Convexity. Finally, we specialize our general results to both sparse Gaussian (least-squares) and Poisson ((generalized) Kullback–Leibler divergence) regression problems. In particular, we show that our general analysis sharpens existing bounds for the LS setting, while providing an entirely new result for the KL case.
\end{abstract}



\section{Introduction}
Let $\bA\in \R^{M\times N}$, $M, N\in\mathbb{N}$, be a linear operator (typically wide, with $M \ll N$) and let $\bx^*$ be a sparse vector with $k^* \ll N$ nonzero components. We consider the inverse problem of recovering $\bx^*$ from the measurements $\by\in\mathcal{Y}^M$ given by
\begin{equation}\label{eq:model}
    \by=\mathcal{R}(\by^*), \text{ with } \by^*=\iota(\bA\bx^*),
\end{equation}
where the --possibly nonlinear-- map $\iota$ represents a deterministic transformation (i.e., the sigmoid function in logistic regression), $\mathcal{R}$ stands for a random operation such as additive (Gaussian) or signal-dependant (Poisson) noise, or a probabilistic ``labeling'' operation as in logistic regression. Finally, the set $\mathcal{Y}\subset \R$ depends on the problem we are dealing with. For instance, $\mathcal{Y}=\R$ in Gaussian regression, $\mathcal{Y}=\{0, 1\}$ in logistic regression, and $\mathcal{Y}=\mathbb{Z}_{\geq 0}$ in Poisson regression. In order to tackle the above problem, it is common to consider the following $\ell_0$-regularized problem
\begin{equation}\label{eq:J0intro}
\min_{\bx\in\cC^N} J_0(\mathbf{x}); \quad J_0(\mathbf{x}):=G_{\by}(\mathbf{A}\mathbf{x}) + \lambda_0\|\mathbf{x}\|_0,
\end{equation}
where $G_{\by}$ corresponds to a general convex data fidelity term (or loss function), and ensures that the solution remains close to the available measurements. Famous examples are the least-squares (or quadratic) loss (associated to Gaussian regression problems), the generalized Kullback--Leibler divergence (common in Poisson regression problems), or the logistic loss (a natural choice for logistic regression problems). The notation $\|\cdot\|_0$ stands for the so-called $\ell_0$ pseudo-norm (or $\ell_0$-norm for simplicity), and measures the amount of nonzero components of a given vector. The scalar $\lambda_0>0$, known as regularization parameter, aims to balance the relative importance between fitting the data and imposing sparsity of solutions, induced by the $\ell_0$-norm. Finally,  we consider in this work problems for which the constraint set $\cC$ is either $\R$ or $\R_{\geq 0}$, and this depends on the fidelity term that we choose. 

Due to the discontinuous and non-convex nature of the $\ell_0$-norm, these problems are inherently NP-hard \cite{Natarajan1995,Nguyen2019}. As a result, a common strategy is to replace the $\ell_0$-norm with continuous relaxations, as they yield more tractable optimization problems. Among these, the $\ell_1$-norm is particularly popular, as its convexity permits the use of fast and efficient optimization algorithms \cite{Tibshi1996}. Beyond the $\ell_1$ convex relaxation, a variety of non-convex alternatives have been proposed. Notable examples include the capped-$\ell_1$ penalty \cite{Zhang2008}, the $\ell_p$-norms, $0<p<1$ \cite{Foucart2009}, and the minimax concave penalty (MCP) \cite{Zhang2010}. More recently, non-convex \emph{exact} continuous relaxations have been proposed~\cite{bian2020smoothing,Carlsson2019,Essafri2024,Soubies2015,soubies2017unified}. These relaxations enjoy the following theoretical property: it can be shown that the relaxed problem shares the same global minimizers as the original $\ell_0$-regularized formulation \cref{eq:J0intro}, while admitting fewer local minimizers. As a result, these formulations offer a more favorable optimization landscape, potentially improving the effectiveness of numerical algorithms.

In particular, as we are interested in dealing with fidelity terms that may be non-quadratic, we will focus on the class of $\ell_0$-Bregman relaxations (B-rex), which has been recently introduced in \cite{Essafri2024}. The corresponding exact relaxation takes the form
\begin{equation}\label{eq:JPsi}
    \min_{\bx \in \cC^N} J_\Psi(\mathbf{x}); \quad J_\Psi(\mathbf{x}):= G_{\mathbf{y}}(\mathbf{A}\mathbf{x})+B_\Psi(\mathbf{x}),
\end{equation}
where, given $\Psi:\R^N\to\R$, $B_{\Psi}$ is known as the B-rex functional, and is associated to the Bregman divergence of the function $\Psi$. We will dedicate~\Cref{sec:brex} to further delve into the details of the above functional: from its definition to its exact relaxation property. In short, the strength of this approach is, on top of the exact relaxation result, its adaptability through the choice of $\Psi$: : a proper choice of the Bregman function $\Psi$ will allow to incorporate the geometry of the data fidelity term $G_\by$ into the regularization imposed by $B_\Psi$, reducing the number of local minimizers, see e.g.~\cite{EssafriKL}.

The case of a quadratic fidelity term $G_{\mathbf{y}}$ and $\Psi=(1/2)\|\cdot\|_2^2$ has been extensively studied in the literature, and the exact relaxation properties of the associated relaxed problem have been analyzed in \cite{Soubies2015,soubies2017unified,Carlsson2019}. Moreover, a deeper analysis of the relaxation's optimization landscape  was conducted in \cite{Carlsson2020}. This work provides conditions under which critical points are isolated in terms of sparsity, meaning that any other critical point admits a much larger cardinality. The authors further demonstrate that, under suitable assumptions, it is possible to determine an interval for the regularization parameter $\lambda_0$ such that the oracle solution is the unique global minimizer of the relaxation and, additionally, isolated in terms of sparsity. Subsequent works analyzed the optimization landscape of exact relaxations for rank-regularized least-squares \cite{carlsson2022unbiased} and for $\ell_0$-$\ell_1$ regularized least-squares \cite{Carlsson22}. Finally, the results of \cite{Carlsson2020} have been specialized to the case where $\bA$ is a Fourier measurement matrix in \cite{Qiao2022}.\\

\textbf{Contributions and outline.} In this work, we analyze the optimization landscape of $J_\Psi$ in \cref{eq:JPsi} for general --not necessarily quadratic-- data fidelity terms. First, we introduce in~\Cref{sec:brsc} the Bregman Restricted Strong Convexity (BRSC) property, that will be central in our manuscript and, as we will see, naturally generalizes well-known quantities such as the Lower Restricted Isometry Property (LRIP)~\cite{Blanchard2011} (which is basically a lower estimate of the so-called RIP~\cite{Candes2006}), or the Restricted Strong Convexity (RSC) property \cite{Shen2016}. In particular, we show that, while the KL data fidelity term fails to satisfy the RSC property, it satisfies the proposed BRSC property for a suitable choice of Bregman generating function. Next, in~\Cref{sec:sec4} we provide sufficient conditions for a critical point of $J_\Psi$ to be isolated in terms of sparsity.  Additionally, recalling that $J_\Psi$ is an exact continuous relaxation of $J_0$ \cite{Essafri2024}, our setting allows us to show that such an isolated critical point turns out to be the unique global minimizer of $J_\Psi$ (and so of $J_0$) for suitable~$\lambda_0$. In~\Cref{sec:oracsol}, we study exact support recovery results. Specifically, we derive sufficient conditions, that can be stated as
$$
\lambda_0\in(\underline\Lambda, \bar\Lambda),
$$
under which the oracle solution $\bx^{\mathrm{or}}$ is both isolated in terms of sparsity and the unique global minimizer of $J_\Psi$ (hence the global minimizer of $J_0$ too). In~\Cref{sec:applications}, we specialize these results to two important instances of Model~\cref{eq:model}, namely Gaussian and Poisson regression, defined respectively in~\Cref{ex:LS} and~\Cref{ex:KL} below. For these two examples we have the following outcomes. 
\begin{enumerate}
\item In the Gaussian regression (least-squares) setting, our analysis allows to obtain improved bounds with respect to existing literature \cite{Carlsson2020}, in the sense that we obtain a larger interval $(\underline\Lambda, \bar\Lambda)$ with respect to \cite{Carlsson2020}. 

\item In the Poisson regression ((generalized) Kullback-Leibler) setting, we provide a completely novel result by constructing a valid interval $(\underline\Lambda, \bar\Lambda)$ of regularization parameters~$\lambda_0$ in which our result holds. Naturally, we see that the same interpretation as in Gaussian regression problems can be obtained: the derived interval enlarges whenever the level of noise is small enough with respect to the magnitude of the nonzero components of the true solution. In this sense, we show that the asymptotic behavior that one would expect holds theoretically too.\\
\end{enumerate}

\textbf{Examples.} We now describe the two examples mentioned above, which consist in considering Model~\cref{eq:model} and varying the deterministic transformation $\iota$ and the random noise described by $\mathcal{R}$. 

\begin{example}[Gaussian regression]\label{ex:LS} Let $\iota$ be the identity mapping, $\iota=\mathrm{Id}$, and let $\mathcal{R}$ denote an additive Gaussian noise. With this, we rewrite Model~\cref{eq:model} as
$$
\by=\bA\bx^*+\beps,
$$
where $\beps\sim\cN(\mathbf{0},\sigma^2\bid)$, $\sigma>0$ being its standard deviation, and $\by\in\cY^M=\R^M$. The inverse problem corresponds to a sparse Gaussian regression problem, where the negative log-likelihood reduces to the data fidelity term being the least-squares loss, which is given for all $\mathbf{w}\in\R^M$ by
\begin{equation}\label{eq:fidLS}
G^{\mathrm{LS}}_\by(\mathbf{w}):=\frac12\|\mathbf{w}-\by\|_2^2.
\end{equation}
Finally, and for further purposes, we note that $G^{\mathrm{LS}}_\by$ has Lipschitz gradient with Lipschitz constant $L=1$.
\end{example}

\begin{example}[Poisson regression]\label{ex:KL} Let $\iota=\mathrm{Id} + \mathbf{b}$, $\mathbf{b} \in \R_{>0}^M$ modeling some known background~\cite{fessler1998paraboloidal}, and let $\mathcal{R}=\cP$ denote a Poisson noise. With this, we rewrite Model~\cref{eq:model} as
$$
\by \sim \mathcal{P}(\by^*),\; \text{with } \;\by^*=\bA\bx^* +\mathbf{b} .
$$
where $\bx^*\in \R_{\geq 0}^N$, $\bA \in \R_{\geq 0}^{M \times N}$, with no column entirely zero, and $\by\in\cY^M=\mathbb{Z}^M_{\geq 0}$  Here, we aim to solve a sparse Poisson regression problem where, as above, the negative log-likelihood reduces to the fidelity term given, for all $\mathbf{w}\in\R_{\geq 0}^M$, by
\begin{equation}\label{eq:fidKL}
G^{\mathrm{KL}}_\by(\mathbf{w}):=\sum_{j=1}^M g_{y_j}^{\mathrm{KL}}(w_j+b_j),
\end{equation}
with $g_y^{\mathrm{KL}}(z):= z+y\log(y/z)-y$  the (generalized) Kullback--Leibler divergence. As above, notice that $G^{\mathrm{KL}}_\by$ has Lipschitz gradient with constant $L = \max_j{y_j}/{b_j^2}$, see e.g.~\cite[Lemma 1]{Harmany2012}.
\end{example}

\textbf{Notation, definitions and assumptions.} In the following, $\R$ and $\R_{\geq 0}$ will denote the spaces of real and non-negative real numbers. Given a vector $\bx$, $\sigma(\bx)$ denotes its support; that is, $\sigma(\bx):=\{i \colon x_i\neq 0\}$. Additionally, given a set of indices $\omega$, we denote by $\# \omega$ its cardinality and, given $\bx$, $\bx_\omega=(x_{\omega[1]},\ldots,x_{\omega[\# \omega]})$ denotes the restriction of $\bx$ to the entries indexed by $\omega$. In this context, given $\bA$ a matrix, $\bA_\omega$ defines the matrix restricted to the columns indexed by $\omega$: $\bA_\omega:=(\ba_{\omega[1]},\ldots,\ba_{\omega[\# \omega]})$, where $\ba_{\omega[j]}$, $j=1,\ldots,\# \omega$, denotes the $\omega[j]^{\mathrm{th}}$ column of $\bA$. We denote by $\bid$ the identity matrix, whose dimension should be clear from the context. Finally, we use the Loewner order for matrices: we say that $\bA\succcurlyeq 0$ (resp. $\bA\succ 0$) if $\bA$ is positive semidefinite (resp. positive definite).  We further recall some useful definitions.
\begin{definition}[Zero padding~\cite{Nikolova2013}] The zero padding operator  $Z_\omega\colon\R^{\#\omega}\to\R^N$ is defined, for any $\bu\in\R^{\#\omega}$, as
$$
(Z_\omega(\bu))_i:=
\begin{cases}
u_j, & \text{for the unique } j\in \{1,\ldots,\#\omega\} \text{ such that } i=\omega[j],\\
0, & \text{if } i\notin \omega.
\end{cases}
$$
\end{definition}
In short, $Z_\omega$ embeds a vector of smaller dimension into a larger one by padding it with zeros outside the index set $\omega$. For example, in $\R^6$, given $\bu=(1,2,3)\in\R^3$ and $\omega=\{2,4,5\}$, we have that $Z_{\omega}(\bu)=(0,1,0,2,3,0)\in\R^6$.
\begin{definition}[Restricted function] Given $F:\R^N\to\R\cup\{+\infty\}$, we define for any $\omega\subseteq\{1,\ldots, N\}$ and any $\mathbf{v}\in\R^{N-\#\omega}$ the restricted function of $F$, $F_{\omega, \mathbf{v}}$, as
$$
F_{\omega,\mathbf{v}}\colon\R^{\#\omega}\to \R\cup\{+\infty\}; \quad  F_{\omega,\mathbf{v}}(\bu):=F(Z_\omega(\bu)+Z_{\omega^c}(\mathbf{v})).
$$
\end{definition}
The above definition corresponds to the restriction of a function $F$ to the variables indexed by $\omega$, the other being fixed to the values in $\mathbf{v}$. Additionally, we simply write $F_\omega$ instead of $F_{\omega,\boldsymbol{0}}$ whenever $\mathbf{v}=\boldsymbol{0}$. Now, we recall the definition of the (symmetric) Bregman divergence.
\begin{definition}[(Symmetric) Bregman divergence] Given a proper, convex and differentiable function $F:\R^N\to\R\cup\{+\infty\}$, we define the Bregman divergence associated to $F$, $D_F$, to be
$$
D_F(\bx, \bx'):= F(\bx)-F(\bx')-\langle\nabla F(\bx'), \bx-\bx'\rangle,
$$
for any $\bx\in\dom F$, $\bx'\in\mathrm{int}(\dom F)$, and $+\infty$ elsewhere, and where $\nabla F(\bx)$ denotes the gradient of $F$ at $\bx$. Moreover, the symmetric Bregman divergence of $F$, denoted by $D^{\ \mathrm{symm}}_F$, is defined as
$$
D_{F}^{\ \mathrm{symm}}(\bx, \bx'):=D_{F}(\bx, \bx')+D_{F}(\bx', \bx)=\langle\nabla F(\bx')-\nabla F(\bx), \bx'-\bx\rangle,
$$
for any $\bx$, $\bx'\in\mathrm{int}(\dom F)$, and $+\infty$ elsewhere.
\end{definition}
Notice that, in the above definitions, $\bx'$ (and $\bx$ too in the symmetric case) must belong to $\mathrm{int}(\dom F)$ as the gradient of $F$ may not necessarily be defined in the border of the domain. 

Next, we introduce a general set of assumptions for the --not necessarily quadratic-- data fidelity terms $G_\by$ that are considered in this work. 

\begin{assumption}\label{ass:FID} The data fidelity term $G_{\by}$ is assumed to be coordinate-wise separable; i.e., given $\by\in\cY^M$,
$$
G_{\by}\colon (\dom g_{y_1}) \times \cdots \times (\dom g_{y_M})\to\R_{\geq 0}; \quad G_{\by}(\mathbf{w}):=\sum_{j=1}^M g_{y_j}(w_j).
$$
where $g_y\colon \R\to\R_{\geq 0}$ is a nonnegative function such that, for any $y\in\mathcal{Y}$, 
\begin{enumerate}[label=(\roman*),leftmargin=*]
\item\label{ass:fid1} $g_y$ is strictly convex, proper, twice differentiable on $\mathrm{int}(\dom g_y)$, and coercive,
\item\label{ass:fid2} $g_y$ has $L$-Lipschitz gradient for some $L>0$.
\end{enumerate}
Moreover, when $\dom G_{\by} \subset \R^M$ we assume that
\begin{enumerate}[label=(\roman*),leftmargin=*]\addtocounter{enumi}{2}
\item\label{ass:fid3} $\mathrm{cl}(\mathrm{Im}(\bA |_{\cC^N})) \subseteq \mathrm{int}(\dom G_{\by})$, where $\mathrm{Im}(\bA |_{\cC^N}) := \{\bA \bx : \bx \in \cC^N\}$,
\item\label{ass:fid4} $\sup_{\mathbf{w}\in \mathrm{Im}(\bA |_{\cC^N}) } \|\nabla G_\by(\mathbf{w}) \| < \infty$.
\end{enumerate} 
\end{assumption}
The first assumption is the one employed in \cite{Essafri2024} to guarantee that~\cref{eq:J0intro} admits solutions and ensure that~\cref{eq:JPsi} is an exact continuous relaxation of~\cref{eq:J0intro}. The second assumption is common in the optimization literature to ensure convergence of several iterative algorithms. Importantly, both assumptions are naturally satisfied by most loss functions of interest and, in particular, by both of the fidelity terms in~\cref{ex:LS} and~\cref{ex:KL}. Finally, note that both the third and the fourth assumptions are satisfied for our KL example as $\mathrm{cl}(\mathrm{Im}(\bA |_{\cC^N})) \subseteq \R_{\geq 0}^M \subset (-\min_{j=1,\ldots, M}b_j,\infty)^M \subseteq \mathrm{int}(\dom G^\mathrm{KL}_{\by})$ and $\lim_{w \to \infty} (g_{y_j}^\mathrm{KL})'(w) = 1$ for all $j=1,\ldots,M$.

We now dedicate the following section to present the central tool of our work, which we call the Bregman Restricted Strong Convexity (BRSC) property and, as we will see, extends existing concepts in the literature.

\section{The Bregman Restricted Strong Convexity Property}\label{sec:brsc}
 
We recall that the objective of this work is to show that the functional $J_\Psi$ defined in~\cref{eq:JPsi}, although non-convex, exhibits a favorable optimization landscape. As we already mentioned, the quadratic setting has been studied in depth in \cite{Carlsson2020}. This work is strongly based on a very specific assumption on the matrix $\bA$: the so-called LRIP condition, introduced in \cite{Blanchard2011}. In the following, we aim to extend such definition to our Bregman setting. 

\subsection{From LRIP/RSC to BRSC}
We start by giving a proper definition of the LRIP condition: we say that $\bA$ satisfies the LRIP property if there exists a constant $\delta_K^-<1$, for some $K \in \{1,\ldots, N\}$, such that
$$
1-\delta_K^-=\inf\left\{\frac{\|\bA\bx\|_2^2}{\|\bx\|_2^2} \mid \bx\neq\boldsymbol{0} \text{ and } \|\bx\|_0\leq K\right\}
$$
or, equivalently, that
\begin{equation}\label{eq:LRIP}
\tag{LRIP}
    \bA_\omega^T\bA_\omega\succcurlyeq (1-\delta_K^-) \bid,\quad \text{ for all } \omega \text{ with } \# \omega\leq K.
\end{equation}
For instance, when $\bA$ is a Gaussian matrix, a common choice in the compressed sensing literature, then~\cref{eq:LRIP} holds with ``overwhelming'' probability~\cite{CandesTao2005}. Since this condition is closely associated with quadratic fidelity terms, it is natural to seek for a corresponding definition that extends existing analyses to general non-quadratic fidelity terms. To this aim, a first approach would be to consider the so-called Restricted Strong Convexity (see~\cite{Shen2016} and references therein): given a proper, convex and differentiable function $F\colon\R^N\to\R\cup\{+\infty\}$, we say that it satisfies the RSC property if there exists a constant $\mu_K>0$ such that, for any $\bx\in\dom F$ and $\bx'\in\mathrm{int}(\dom F)$ with $\|\bx-\bx'\|_0\leq K$, we have
$$
D_F(\bx,\bx')\geq\frac{\mu_K}{2}\|\bx-\bx'\|_2^2,
$$
which is equivalent to having, for any $\bx, \bx'\in\mathrm{int}(\dom F)$ with $\|\bx-\bx'\|_0\leq K$:
\begin{equation}\label{eq:RSC}
\tag{RSC}
D_F^{\ \mathrm{symm}}(\bx,\bx') \geq \mu_K\|\bx-\bx'\|_2^2.
\end{equation}
Notice that both of the above expressions are simply equivalent definitions of strong convexity of $F$ with the additional assumption that the distinct points must satisfy $\|\bx-\bx'\|_0\leq K$. Additionally, the above condition trivially recovers the LRIP property on the matrix $\bA$, whenever $F=G_\by^{\mathrm{LS}}(\bA\cdot)$, by observing that  can be equivalently defined as follows: we say that $F$ satisfies the RSC property whenever the quantity
$$
   \mu_K:=\inf\left\{\frac{D_{F}^{\ \mathrm{symm}}(\bx', \bx)}{\|\bx'- \bx\|_2^2} \mid  \bx\neq\bx'  \text{ and }\|\bx'-\bx\|_0\leq K \right\} 
$$
is positive. The relation between both constants $\mu_K$ and $\delta_K^-$ when $F=G_\by^{\mathrm{LS}}(\bA\cdot)$ is $\mu_K=1-\delta_K^-$. In fact, it is possible to relate~\cref{eq:RSC} and~\cref{eq:LRIP} beyond the least-squares data term.
\begin{proposition}\label{propo:BRSC_with_psiL2} 
Let $G\colon\R^N\to\R\cup\{+\infty\}$ be a proper, differentiable, and strongly convex function with parameter $\nu>0$. Assume in addition that $\bA$ satisfies~\cref{eq:LRIP} with constant $\delta_K^-<1$. Then, $F = G(\bA\cdot)$ satisfies~\cref{eq:RSC} with constant $\mu_K:=\nu(1-\delta_K^-)$.
\end{proposition}
We point out that the above result is natural, and has already been proved in the full row rank case in \cite[Theorem 1]{Zhang2014} (notice that, in our setting, requiring \cref{eq:LRIP} corresponds to $\bA_\omega$, for $\# \omega \leq K$, being full column rank). For clarity, we dedicate~\Cref{proof:propoBRSC_Psil2} to prove the result. 

However,~\cref{eq:RSC} turns to be irrelevant for data fidelity terms $F = G_\by (\bA \cdot)$ where $G_\by$ is not strongly convex. This is the case, for instance, for the (generalized) Kullback--Leibler divergence $G_\by^{\mathrm{KL}}$ from~\cref{ex:KL}. To see this, as it is separable, it suffices to consider the one dimensional case: given $y\in\mathbb{Z}_{\geq 0}$, we have that $(g_y^{\mathrm{KL}})''(z)=y/(z+b)^2$ tends to zero as $z\to\infty$, that is $\nu = 0$ and thus $\mu_K =0$ for any $K$ in~\Cref{propo:BRSC_with_psiL2}. Motivated by this observation, we introduce in this work what we call the \emph{Bregman Restricted Strong Convexity} (BRSC) property which, we anticipate, allows to obtain positive results for $G_\by^{{\mathrm{KL}}}(\bA\cdot)$ (cf.~\cref{propo:positiveKL}).

\begin{definition}\label{def:brsc} Let $F, \Phi\colon\R^N\to\R \cup\{+\infty\}$ be proper, convex and differentiable functions over the interior of their domains. We say that $F$ satisfies the \emph{Bregman Restricted Strong Convexity (BRSC)} property with respect to $\Phi$, $\mathcal{X} \subseteq \mathrm{int}(\dom F) \cap\mathrm{int}(\dom \Phi)$, and $K \leq N$, if there exists a constant $C_K>0$ such that, for any $\bx \in \mathcal{X}$, $\bx'\in \mathrm{int}(\dom F)\cap\mathrm{int}(\dom \Phi)$ with $\bx'\neq\bx$ and $\|\bx-\bx'\|_0\leq K$, we have
\begin{equation}\label{eq:brsc}
    D_{F}^{\ \mathrm{symm}}(\bx, \bx')\geq C_K D_{\Phi}^{\ \mathrm{symm}}(\bx, \bx').
    \tag{BRSC}
\end{equation}
\end{definition}
\begin{remark}
Note that, although we use the notation $C_K$, this constant also depends on $\cX$, $F$, and $\Phi$. However, for a given instance of Problem~\cref{eq:J0intro}, these elements are fixed, whereas we may be interested in analyzing different values of~$K$. 
\end{remark}

The above definition will serve as the main technical tool in this work, and can be seen as a restricted version of the $\Phi$-strong convexity property in \cite{Bartlett2007} or the relative strong convexity in \cite{Lu2018}. 

We now provide some additional remarks to clarify its implications. First, observe that it generalizes existing notions; for instance, the choices $\Phi = (1/2) \|\cdot\|_2^2$ and $\mathcal{X} = \mathrm{int}(\dom F)$ recover~\cref{eq:RSC} with $\mu_K = C_K$ and, consequently~\cref{eq:LRIP} as well (with $1-\delta_K^-=C_K$), whenever $F=G_\by^{\mathrm{LS}}(\bA\cdot)$. In the following section, we present some properties of our BRSC, that will be useful throughout the manuscript.

\subsection{Properties of BRSC}
To begin with, we show in the following proposition an interesting monotonicity property of our BRSC, which also holds for both LRIP and RSC. Roughly, we prove that the smaller the set $\cX$ and/or $K$ is, the larger the BRSC constant.
\begin{proposition}[Monotonicity of BRSC]\label{prop:orderBRSCconst}
    If $F$ satisfies~\cref{eq:brsc} with respect to $\Phi$, $\cX$, and $K$, then it also satisfies~\cref{eq:brsc} with respect to $\Phi$, any $\tilde{\mathcal{X}} \subseteq \mathcal{X}$, and any $\tilde K \leq K$. Moreover, we have $C(\tilde K,\tilde{\mathcal{X}},F,\Phi) \geq C(K,\mathcal{X},F,\Phi)$.
\end{proposition}
\begin{proof}
    The results follows from the fact that, for  any $\tilde{\mathcal{X}} \subseteq \mathcal{X}$ and any $\tilde{K} \leq K$ we have
    \begin{multline*}
        \left\lbrace \bx, \bx' | \bx \in \tilde{\mathcal{X}}, \bx' \in  \mathrm{int}(\dom F)\cap\mathrm{int}(\dom \Phi), \|\bx - \bx'\|_0 \leq \tilde K \right\rbrace \\ \subseteq  \left\lbrace \bx, \bx' | \bx \in {\mathcal{X}}, \bx' \in  \mathrm{int}(\dom F)\cap\mathrm{int}(\dom \Phi), \|\bx - \bx'\|_0 \leq  K \right\rbrace
    \end{multline*}
    which are the sets over which the~\cref{eq:brsc} inequality needs to hold true. 
\end{proof}

In the following result we consider restrictions over a subset of variables indexed by $\omega\subseteq\{1,\ldots, N\}$.
\begin{proposition}\label{prop:implications_BRSC}
    Let $F$ satisfy~\cref{eq:brsc} with respect to $\Phi$, a convex $\cX$, and $K$. Let $\omega \subseteq \{1,\ldots,N\}$ such that $\#\omega \leq K$ and define $\mathcal{X}_\omega = \{\bu \in \R^{\#\omega} | {Z}_\omega(\bu) \in \mathcal{X}\}$. Then, the restrictions $F_\omega$ and $\Phi_\omega$ are such that 
    \begin{enumerate}[label=(\roman*),leftmargin=*]
        \item\label{eq:brsc_implies_gcvx} $h(\bu) := F_\omega(\bu) - C_K\Phi_\omega(\bu)$ is convex over $\mathcal{X}_\omega$.
        \item\label{eq:brsc_implies_con_ineq} For all $ \bu,\bu' \in \mathcal{X}_\omega$, we have that 
        $$
        F_\omega(\bu) \geq F_\omega(\bu') + \left< \nabla F_\omega(\bu'),\bu - \bu'\right> + C_K D_{\Phi_\omega}(\bu,\bu').
        $$
    \end{enumerate}
\end{proposition}
\begin{proof}
    First of all, by its definition, $\mathcal{X}_\omega$ is convex as $\mathcal{X}$ is.
    Then, noticing that $\nabla h(\bu)=\nabla F_\omega(\bu) - C_K\nabla\Phi_\omega(\bu)$, we get from the definition of the symmetric Bregman divergence and~\cref{eq:brsc} that, for all $\bu,\bu' \in \mathcal{X}_\omega$,
    $$
    D_h^{\ \mathrm{symm}}(\bu, \bu')=\left<(\nabla F_\omega(\bu) - C_K \nabla\Phi_\omega(\bu)) - (\nabla F_\omega(\bu') - C_K\nabla \Phi_\omega(\bu')), \bu - \bu' \right> \geq 0,
    $$
    which proves~\cref{eq:brsc_implies_gcvx} (monotone gradient condition). Then, given that $h$ is convex over $\mathcal{X}_\omega$, we have that, for all $\bu,\bu' \in \mathcal{X}_\omega$,
    $$
    h(\bu) \geq h(\bu') + \left<\nabla h(\bu'), \bu - \bu' \right>.
    $$
    Equivalently,
    $$
    F_\omega(\bu) \geq F_\omega(\bu') +   \left<\nabla F_\omega(\bu'), \bu - \bu' \right> + C_K \left(\Phi_\omega(\bu) - \Phi_\omega(\bu') -\left< \nabla \Phi_\omega(\bu'), \bu - \bu' \right> \right),
    $$
    which completes the proof by definition of $D_{\Phi_\omega}$.
\end{proof}

\subsection{BRSC and the Kullback--Leibler divergence}
To start this section, let us emphasize that our~\cref{eq:brsc} generalizes~\cref{eq:RSC} in two ways. First, it allows a better adaptation to the  geometry of the problem through an appropriate choice of~$\Phi$. Second, it requires the inequality to hold only locally with respect to one of the two variables through the set $\cX$. In fact, in the following we will show that both of these considerations; i.e., $1)$ the choice of $\Phi$ and $2)$ the presence of the subset $\cX\subseteq\cC^N$, turn to be highly relevant for obtaining positive BRSC constants for the Poisson regression problem described in~\Cref{ex:KL}. We start with the following proposition, in which we set the Bregman-generating function $\Phi$ to be the (halved) square norm.

\begin{proposition}\label{propo:negativeKL} Let $\Phi=(1/2) \|\cdot\|_2^2$ and let $\cX\subset\cC^N$ be a compact set. Then, for $G_\by^{\mathrm{KL}}(\bA \cdot)$, we have $C_K = 0$ for any $K \geq 1$.
\end{proposition}

We defer its proof to~\Cref{proof:negativeKL}. In short, by the monotonicity property given in~\cref{prop:orderBRSCconst}, the proof consists in showing that $C_1=0$. The above result demonstrates that, if $\Phi=(1/2)\|\cdot\|_2^2$, even when restricting the first variable of the \cref{eq:brsc} (equivalently \cref{eq:RSC} here) inequality to a compact set $\mathcal{X}$, the best constant is $0$ for $G_\by^{\mathrm{KL}}(\bA \cdot)$. Therefore, the following question arises naturally: what if we consider instead a Bregman-generating function $\Phi$ better adapted to the data term? In the following result, we will actually prove the existence of a nonzero constant $C_K$ for a well chosen Bregman-generating function $\Phi$ and any compact set $\mathcal{X} \subseteq \mathrm{int}(\dom F) \cap \mathrm{int}(\dom \Phi)$.

\begin{theorem}\label{propo:positiveKL}
Let $\Phi$ be the (smoothed) Burg entropy,
\begin{equation}\label{eq:burg}
\Phi(\bx):= \sum_{i=1}^N -\log(x_i + \eta_i),
\end{equation}
for some $\eta_i>0$, $i=1,\ldots, N$, and assume that the matrix $\tilde\bA\in \R_{\geq 0}^{\#\sigma_\by \times N}$, which is the restriction of $\bA$ to the rows indexed by $\sigma_\by := \sigma(\by)$, satisfies~\cref{eq:LRIP} with constant $\delta_K^-<1$. Assume in addition that $\cX$ is a compact set. Then, there exists $C_K>0$ such that~\cref{eq:brsc} holds for $G_\by^{\mathrm{KL}}(\bA \cdot)$.
\end{theorem}

We defer its proof to~\Cref{proofthmKL}. To conclude, we point out that, in this case,~\cref{eq:LRIP} turns to be a sufficient condition for~\cref{eq:brsc} to hold true. Additionally, a precise expression of the constant $C_K$ is given in the proof. It depends on both $\delta^-_K$ and the min/max of some quantities over the compact set $\cX$. Therefore, in some sense, finding BRSC constants is as difficult as finding LRIP constants. Finally, a natural question is whether~\cref{eq:LRIP} is sufficient for any data term $G_\by$, not only the KL divergence, to satisfy~\cref{eq:brsc} with a suitable choice of $\Phi$. In a future work, we plan to study the latter question and other consequences related to BRSC.

\section{\texorpdfstring{The $\ell_0$-Bregman relaxation (B-rex)}{The l0-Bregman relaxation}}\label{sec:brex}

In this section, we recall the results of \cite{Essafri2024}, namely the definition of the $\ell_0$-Bregman relaxation, some relevant properties, and the associated exact relaxation result. 

\begin{definition}\label{ass:generating_func_psi} Let $\{\psi_i\}_{i=1}^N$, $\psi_i\colon\R\to\R$ with $\cC\subseteq \dom \psi_i$ for all $i=1,\ldots, N$, be a family of scalar functions such that
\begin{enumerate}[label=(\roman*),leftmargin=*]
    \item\label{ass1psi} $\psi_i$ is strictly convex, proper, and twice differentiable over $\mathrm{int}(\cC)$,
    \item\label{ass2psi} the map $x\mapsto \psi_i'(x)x-\psi_i(x)$ is coercive,
    \item\label{ass3psi} when $\cC=\R$, $\psi_i'$ is odd and, when $\cC=\R_{\geq 0}$, $\psi'_i(0)\geq 0$.
\end{enumerate}
The Bregman-generating function $\Psi$ is defined, for any $\bx\in\cC^N$, as
$$
\Psi\colon\cC^N\to\R, \quad \Psi(\bx):=\sum_{i=1}^N \psi_i(x_i).
$$
Then, given $\lambda_0>0$, the $\ell_0$-Bregman relaxation (B-rex) associated to $\Psi$ is defined, for any $\bx\in\cC^N$, as
$$
B_\Psi(\bx):=\sup_{\rho\in\R}\sup_{\bx'\in\mathrm{int}(\cC^N)} \{\rho-D_\Psi(\bx, \bx') \ : \ \rho-D_\Psi(\cdot, \bx')\leq \lambda_0\|\cdot\|_0\},
$$
where, since $\Psi$ is separable, $D_\Psi$ can be written as the sum of one-dimensional Bregman divergences $d_{\psi_i}$; i.e., for any $\bx\in \cC^N$ and any $\bx'\in\mathrm{int}(\cC^N)$, 
$$
D_\Psi(\bx,\bx')=\sum_{i=1}^N d_{\psi_i}(x_i, x'_i),
$$
being $d_{\psi_i}(x,x'):=\psi_i(x)-\psi_i(x')-\psi_i'(x')(x-x')$ for all $x\in\cC$, $x'\in\mathrm{int}(\cC)$.
\end{definition}

Note that both~\cref{ass1psi} and~\cref{ass2psi} in~\cref{ass:generating_func_psi} are the same as in \cite{Essafri2024}. In particular,~\cref{ass2psi} has been included so that the $\lambda_0$-sublevel sets of each $d_{\psi_i}(0, \cdot)$ are bounded convex sets, and so of the form $[-\alpha_i, \alpha_i]$ when $\cC=\R$ or $[0,\alpha_i]$ when $\cC=\R_{\geq 0}$. Here, we drop the superindices $+$ and $-$ of \cite{Essafri2024} as, if $\cC=\R_{\geq 0}$, then simply $\alpha_i^-=0$ and, from~\cref{ass3psi}, we have  $\alpha^-_i=-\alpha^+_i$ whenever $\cC=\R$. Finally, notice that, although~\cref{ass3psi} is specific to this work, it is satisfied by each of the Bregman-generating functions $\psi_i$ considered in \cite[Table 2]{Essafri2024}. Finally, $B_\Psi$ is separable, $B_\Psi(\bx)=\sum_{i=1}^N\beta_{\psi_i}(x_i)$, where each $\beta_{\psi_i}$ is defined, for all $x \in \cC$, as
\begin{equation}
\beta_{\psi_i}(x):=
\begin{cases}
\psi_i(0)-\psi_i(x)+\mathrm{sign}(x)\psi_i'(\alpha_i)x, & \text{ if } |x|\leq \alpha_i,\\
\qquad \lambda_0, & \text{ if } |x|> \alpha_i.
\end{cases} 
\end{equation}
For further purposes, we recall the closed-form expression of the Clarke's subdifferential of the B-rex penalty. As it is separable, we have that
$$
\partial B_\Psi(\bx)=\partial \beta_{\psi_1}(x_1)\times\cdots\times\partial\beta_{\psi_N}(x_N),
$$
where, when $\cC = \R$, we have
$$
\partial\beta_{\psi_i}(x)=
\begin{cases}
-\psi_i'(x)+\mathrm{sign}(x)\psi_i'(\alpha_i), & \text{if } 0 <|x| \leq \alpha_i,\\
[-\psi_i'(\alpha_i)-\psi_i'(0),\psi_i'(\alpha_i)-\psi_i'(0)], & \text{if } x=0,\\
\qquad 0, & \text{otherwise},
\end{cases}
$$
and, when $\cC = \R_{\geq 0}$,
$$
\partial\beta_{\psi_i}(x)=
\begin{cases}
-\psi_i'(x)+\psi_i'(\alpha_i), & \text{if } x\in[0,\alpha_i],\\
\qquad 0, & \text{otherwise}.
\end{cases}
$$
It turns out that the B-rex relaxed problem~\cref{eq:JPsi} is an exact continuous relaxation of $J_0$ for suitable choices of $\Psi$. 

\begin{theorem}{\cite[Theorem 9]{Essafri2024}}\label{thm:exactrel} Let $\Psi$ be such that the following concavity condition  is satisfied : for any $\bx\in\cC^N$ and $i = 1,\ldots, N$,
\begin{equation} \label{eq:CC}
g(t):=J_{\Psi}(\bx^{(i)}+t\boldsymbol{e}_i) \text{ is strictly concave on } (-\alpha_i, 0) \text{ and } (0, \alpha_i),
\tag{CC}
\end{equation}
where, for any $t\in\R$, $\bx^{(i)}+t\boldsymbol{e}_i:=(x_1,\ldots, x_{i-1}, t, x_{i+1}, \ldots, x_N)$.
Then, $J_\Psi$ is an exact continuous relaxation of $J_0$; i.e., the set of global minimizers of $J_\Psi$ and $J_0$ coincide, and local minimizers of $J_{\Psi}$ are local minimizers of~$J_0$. 
\end{theorem}

\section{\texorpdfstring{Optimization landscape analysis of $\ell_0$-Bregman relaxations}{Optimization landscape analysis of l0-Bregman relaxations}}\label{sec:sec4}

The current section will be divided into three parts. First, we provide a simple characterization of the critical points of $J_\Psi$ that will be useful along the rest of this manuscript. Second, we focus on analyzing the optimization landscape of the relaxed functional $J_{\Psi}$, providing sufficient conditions so that there is a global minimizer of $J_\Psi$ that is (what we call) \emph{isolated in terms of sparsity}. Throughout this part, we fix $F:=G_{\by}(\bA\cdot)$. Moreover, in all our examples we have $\cC^N \subseteq \mathrm{int}(\dom F) \cap \mathrm{int}(\dom \Psi)$ and we will thus simply use $\cC^N$ in place of this intersection in what follows.

\subsection{\texorpdfstring{Characterizing the critical points of $J_\Psi$}{Characterizing the critical points of J-Psi}}

In order to find the critical points of the relaxation $J_\Psi$, we first define, for all $\bx\in\cC^N$, the following auxiliary function:
$$
\mathcal{H}(\bx):=B_\Psi(\bx)+\Psi(\bx),
$$
and observe that, by \cite[Proposition 11]{Essafri2024}, $\mathcal{H}$ is the lower semicontinuous convex envelope of $\lambda_0\|\cdot\|_0+\Psi$; i.e. the largest convex function below $\lambda_0\|\cdot\|_0+\Psi$. It follows that
$$
J_\Psi(\bx)=\mathcal{H}(\bx)-\Psi(\bx)+G_{\by}(\bA\bx).
$$
In this way, $J_\Psi$ is written as the sum of two terms: the convex, non-differentiable function $\mathcal{H}$ (due to the non-differentiability of $B_\Psi$), and the non-convex, differentiable part $-\Psi+G_{\by}(\bA\cdot)$. Critical points of $J_\Psi$ are characterized via first-order optimality conditions: $\bx\in\cC^N$ is a critical point of $J_\Psi$ if and only if
$$
\boldsymbol{0}\in \partial \mathcal{H}(\bx)-\nabla\Psi(\bx)+\bA^T\nabla G_{\by}(\bA\bx),
$$
where $\partial \mathcal{H}(\bx)$ denotes the subdifferential operator of $\mathcal{H}$ at $\bx$. Rearranging the terms, we get that
\begin{equation}\label{eq:optconds}
\nabla\Psi(\bx)-\bA^T\nabla G_{\by}(\bA\bx)\in\partial \mathcal{H}(\bx).
\end{equation}
For further convenience, we will associate to any $\bx$ the left hand side of the formula above by defining the following quantity
\begin{equation}\label{eq:z}
\bz:=\nabla\Psi(\bx)-\bA^T\nabla G_{\by}(\bA\bx),
\end{equation}
and in this way the optimality conditions derived in \cref{eq:optconds} simply write $\bz\in\partial\mathcal{H}(\bx)$. Therefore, we conclude that  $\bx\in\cC^N$ is a critical point for $J_\Psi$ if and only if $\bz\in\partial\mathcal{H}(\bx)$. To prove the upcoming results, we will exploit the monotonicity of the operator $\partial \mathcal{H}$. This is the motivation behind the introduction of the convex $\mathcal{H}$ above, so as to separate the non-differentiability and non-convexity of $J_\Psi$ in two different terms.

Before delving into the next section, and for further purposes, we investigate the subdifferential of $\mathcal{H}$. First, notice that $\mathcal{H}$ is separable: for all $\bx\in\cC^N$ we have $\mathcal{H}(\bx)=\sum_{i=1}^N h_i(x_i)$, with $h_i:\cC\to\R$ defined as $h_i(x):= \beta_{\psi_i}(x)+\psi_i(x)$ for any $x\in\cC$. Therefore, it holds true that, for every $\bx\in\cC^N$,
$$
\partial \mathcal{H} (\bx) = \partial h_1(x_1) \times \cdots \times \partial h_N(x_N),
$$
where, if $\cC=\R$, we have that
\begin{equation}\label{eq:subdiffg_i}
\partial h_i(x)= \partial \beta_{\psi_i}(x)+\psi_i'(x)=
\begin{cases}
\mathrm{sign}(x)\psi_i'(\alpha_i), & \text{if} \quad 0<|x|\leq\alpha_i,\\
[-\psi_i'(\alpha_i), \psi_i'(\alpha_i)], & \text{if} \quad x=0,\\
\psi_i'(x), &  \text{otherwise}.
\end{cases}
\end{equation}
In the constrained case $\cC=\R_{\geq 0}$, we get
$$
\partial h_i(x)=
\begin{cases}
\psi_i'(\alpha_i), & \text{if} \quad 0\leq x\leq\alpha_i,\\
\psi_i'(x), &  \text{otherwise}.
\end{cases}
$$
To conclude, we point out that, as $\partial \mathcal{H}$ is separable, the condition $\bz\in\partial\mathcal{H}(\bx)$ can be equivalently analyzed element-wise: $\bz\in\partial\mathcal{H}(\bx)$, with $\bz$ defined as in \cref{eq:z}, if and only if $z_i\in\partial h_i(x_i)$ for all $i=1,\ldots, N$.
\subsection{Uniqueness of a sparse minimizer with B-rex}\label{sec:uniquesparse}
In this section, we will show that, under certain conditions, $J_{\Psi}$ admits an isolated sparse global minimizer, in the sense that any other critical point has a much larger cardinality. The theorem below will serve as the cornerstone of this work.
\begin{theorem}\label{thm:sparsest} 
    Assume that $F$ satisfies~\cref{eq:brsc} with respect to $\Psi$, $\cX$, $K$, and with constant $C_K >0$. Let $\bx\in\cX$ be a critical point of $J_\Psi$ such that, for every $i=1,\ldots, N$,
    \begin{equation}\label{eq:zinotin}
        |z_i|\notin \left[C_K(\psi_i'(\alpha_i)-\psi_i'(0)) + \psi_i'(0), \frac{\psi_i'(\alpha_i)- \psi_i'(0)}{C_K} + \psi_i'(0)\right],
    \end{equation}
    where $\bz$ is given by \cref{eq:z}. Then, any other critical point $\bx'\in\cC^N$ of $J_\Psi$, $\bx'\neq\bx$, satisfies $\|\bx'-\bx\|_0>K$.
\end{theorem}

The proof is provided in~\Cref{proof:thm41}. \cref{thm:sparsest} states that, under the BRSC property of $F$ with respect to $\Psi$, $\mathcal{X}$, $K$, and with constant $C_K$, if $J_\Psi$ admits a critical point $\bx \in \mathcal{X}$ with $\|\bx\|_0 \leq K/2$ such that~\cref{eq:zinotin} holds, then, by the triangle inequality, it is the sparsest one, in the sense that any other critical point $\bx' \in \cC^N$ has a cardinality $\|\bx'\|_0 > K/2$. Indeed. Observe that
$$
K<\|\bx-\bx'\|_0\leq \|\bx\|_0+\|\bx'\|_0 \leq \frac{K}{2}+\|\bx'\|_0,
$$
which directly implies that $\|\bx'\|_0 > K/2$, as claimed. In particular, if $\|\bx\|_0 \ll K/2$ then any other critical point $\bx'$ is such that $\|\bx'\|_0 \gg K/2$. Therefore, we say that a critical point $\bx$ of $J_\Psi$ that satisfies the assumptions of~\cref{thm:sparsest} is ``isolated'' in terms of sparsity, which indicates a favorable optimization landscape for~$J_\Psi$.  We next comment on the required assumptions in the above result. 

First, notice that the symmetry in condition~\cref{eq:zinotin} regarding the treatment of $z_i$ and $-z_i$ is due to \cref{ass3psi} in \cref{ass:generating_func_psi} which allows us to simplify the presentation. If $\cC=\R$ and the $\psi'_i$ were not odd, as already indicated after~\cref{ass:generating_func_psi}, we first recover the notation $\alpha_i^-$ and $\alpha_i^+$ as in \cite{Essafri2024}  (instead of $\pm\alpha_i$ as we do here), and then consider two intervals instead of only one as in~\cref{eq:zinotin}: one associated to $\alpha_i^+$ (i.e., the one given in~\cref{eq:zinotin}), which is contained in $\R_{\geq 0}$, and another one associated to $\alpha_i^-$, which is contained in $\R_{\leq 0}$. In this way, condition~\cref{eq:zinotin} would translate into requiring that $z_i$ does not belong to the union of both intervals. 

Next, given a critical point $\bx\in \cC^N$ of $J_\Psi$, the BRSC requirement to apply this theorem can be reduced to the singleton $\mathcal{X} = \{\bx\}$. This would lead to better (larger)  BRSC constants (cf.~\cref{prop:orderBRSCconst}), and therefore a less restrictive condition in~\cref{eq:zinotin}, compared to imposing~\cref{eq:brsc} globally over the entire space $\mathcal{X} = \cC^N$. 

Finally, we point out that when $C_K>1$, the  interval in~\cref{eq:zinotin} is the empty set, and so the stated condition is always satisfied. This leads to the following corollary. 
\begin{corollary}
    Assume that $F$ satisfies~\cref{eq:brsc} with respect to $\Psi$, $\cX$, $K$, and with constant $C_K>1$. Then, $J_\Psi$ admits \emph{at most} one critical point $\bx\in\cX$ such that $\|\bx\|_0 \leq K/2$.
\end{corollary}

The following theorem provides a sufficient condition on $\lambda_0$ under which the sparsest critical point of $J_\Psi$, in the sense of~\cref{thm:sparsest}, is  the unique global minimizer of~$J_\Psi$.

\begin{theorem}\label{thm:uniqueglob} 
    Assume that $F$ satisfies~\cref{eq:brsc} with respect to $\Psi$, $\cX$, $K$, and with constant $C_K>0$. Let $\bx\in\cX$ be a critical point of $J_\Psi$ such that~\cref{eq:zinotin} holds. If we have in addition that \cref{eq:CC} is satisfied and
    \begin{equation}\label{eq:uniqglobcond}
       \lambda_0>\frac{F(\bx)}{1+K-2\|\bx\|_0}, 
    \end{equation}
    then $\bx$ is the unique global minimizer of $J_{\Psi}$ (and so of $J_0$), and any other critical point $\bx'\in\cC^N$ of $J_\Psi$, $\bx'\neq\bx$, satisfies $\|\bx'-\bx\|_0>K$.
\end{theorem}

We present the proof in~\Cref{proof_thm:uniqueglob}. 

In this section, we have shown that under suitable conditions, the global minimizer of the exact relaxation $J_\Psi$ is the sparsest critical point and is isolated in terms of sparsity; i.e., any other critical point has a much larger cardinality. It is worth mentioning that such nice properties do not hold for the original function $J_0$. Indeed, for any support $\omega \subseteq \{1,\ldots,N\}$, $J_0$ admits a local minimizer~\cite[Proposition 2]{Essafri2024}. As such, there exist many local minimizers for each sparsity level in $J_0$. Hence, while preserving all global minimizers of $J_0$, the exact relaxation $J_\Psi$ exhibits a more favorable optimization landscape. To conclude, we note that, actually, all the results of the present section hold for any twice differentiable function $F$ that satisfies the conditions given in~\cref{ass:FID}, meaning that $F$ does not necessarily need to be a data fidelity term.

\section{Towards the oracle solution}\label{sec:oracsol}
In this section, we exploit the obtained results to the general problem of recovering the ground truth $\bx^*$ from the noisy measurements $\by$ in \cref{eq:model}. Specifically, we study under which conditions the so-called oracle solution is the global minimizer of both $J_0$ and its exact relaxation $J_\Psi$ and, in addition, is isolated in terms of sparsity for $J_\Psi$. From now on, we will assume that the concavity condition given in \cref{eq:CC} holds; that is, $J_\Psi$ is an exact continuous relaxation of $J_0$. The starting point of this section is the sparse inverse problem described in the introduction: we aim to recover an unknown source $\bx^*\in\cC^N$ from noisy measurements $\by$ given by~\cref{eq:model}. We denote the support of $\bx^*$ by $\sigma^*:=\sigma(\bx^*)$, its length as $k^* = \# \sigma^*$, and we let $\bu^*= \bx^*_{\sigma^*} \in\cC^{k^*}$; i.e., $\bu^*$ denotes the true solution $\bx^*$ restricted to its support. Throughout this section, we set $F=G_\by(\bA\cdot)$ and we assume that the restricted matrix $\bA_{\sigma^*}$ has full rank. Next, we define the \emph{oracle solution} of \cref{eq:model} as $\bx^\mathrm{or}=Z_{\sigma^*}(\bu^{\mathrm{or}})$, with 
\begin{equation}\label{eq:oracsol}
\mathbf{u}^\mathrm{or} = \argmin_{\mathbf{u}\in\cC^{k^*}} F_{\sigma^*}(\bu), 
\end{equation}
where we recall that the restricted function $F_{\sigma^*}$ reads 
$
F_{\sigma^*}(\bu):=G_\by(\bA_{\sigma^*}\bu).
$

Note also that $\bu^\mathrm{or}$ is unique as $F_{\sigma^*}$ is a strictly convex and coercive function because $G_\by$ is strictly convex and coercive, and $\bA_{\sigma^*}$ has full rank. Moreover, as $\bu^{\mathrm{or}}$ is the minimizer of $F_{\sigma^*}$, it holds true that for all $\bu \in \cC^{k^*}$,
\begin{equation}\label{eq:gradFzeor_xor}
\left< \nabla F_{\sigma^*}(\bu^\mathrm{or}) , \bu - \bu^\mathrm{or} \right> \geq 0. 
\end{equation}
Moreover, on the oracle support $\sigma^\mathrm{or} := \sigma(\bx^\mathrm{or}) \subseteq  \sigma^*$ we have $\bu^\mathrm{or} \in \mathrm{int}(\cC^{\sharp \sigma^\mathrm{or}})$, and thus
\begin{equation}\label{eq:zerogradientcond}
\nabla F_{\sigma^\mathrm{or}}(\bu^\mathrm{or}) = \bA_{\sigma^\mathrm{or}}^T\nabla G_{\by}(\bA_{\sigma^\mathrm{or}}\bu^\mathrm{or})=\boldsymbol{0}.
\end{equation}
Finally, by construction,  $\bx^\mathrm{or}$ is a local minimizer of $J_0$~\cite[Proposition 2]{Essafri2024}.

\begin{remark}
   Notice that asking for $\bA_{\sigma^*}$ to be full rank has several implications. First, it serves to define the oracle solution within our context. Additionally, and as we aim to show that the oracle solution is a global minimizer of $J_\Psi$ (and hence of $J_0$ too), there would be no hope to show this if $\bA_{\sigma^*}$ were not full rank. Indeed, by \cite[Theorem 4]{Essafri2024}, all global minimizers of $J_0$ are strict. Moreover, by \cite[Theorem 3]{Essafri2024}, any local minimizer $\bx$ of $J_0$ is  strict if the restricted matrix $\bA_{\sigma(\bx)}$ is full rank. Therefore, if $\bA_{\sigma^*}$ is not full rank, $\bx^{\mathrm{or}}$ cannot be a global minimizer of $J_\Psi$. 
\end{remark}

Next, on top of the concavity condition \cref{eq:CC} that ensure $J_\Psi$ to be an exact relaxation of $J_0$, we make an additional structural assumption on $F$. 

\begin{assumption}\label{ass:BRSCor} We assume that $F$ satisfies~\cref{eq:brsc} with respect to $\Psi$, $K\geq 2k^*$, and a compact set $\cX$ such that $\bx^{\mathrm{or}}, \bx^*\in\cX$. Additionally, we require that its construction does not depend on $\bx^{\mathrm{or}}$. For instance, a valid choice for $\cX$ is
\begin{equation}\label{eq:compactX}
\cX:=\{\bx\in\cC^N  \mid  \mathrm{supp}(\bx)\subseteq \sigma^* \text{ and } F_{\sigma^*}(\bx_{\sigma^*})\leq F_{\sigma^*}(\bx^*_{\sigma^*})\}.
\end{equation}
\end{assumption}
The required conditions on the set $\cX$ in the above assumption are motivated by three underlying considerations. First, the subsequent analysis builds on the results of~\Cref{sec:uniquesparse}, which must be applied at both $\bx^*$ and $\bx^{\mathrm{or}}$. Therefore, $\cX$ is required to contain both points. For instance, if we consider $\cX$ given by~\cref{eq:compactX}, the inclusion of $\bx^*$ is trivial, while that of  $\bx^{\mathrm{or}}$ follows from its definition. Second, it should be as small as possible to get better BRSC constants (cf.~\cref{prop:orderBRSCconst}). In our example, this is achieved with the sublevel-set constraint. Moreover, the compactness assumption is particularly important for guaranteeing nonzero BRSC constants in cases such as the KL fidelity term (cf.~\Cref{propo:positiveKL}). Notice that, by coercivity of $F_{\sigma^*}$, $\cX$ given by~\cref{eq:compactX} is  compact. In contrast, for the quadratic case ($F=G^{\mathrm{LS}}_\by(\bA\cdot)$), this compactness assumption does not affect BRSC constants (as it holds globally) and can therefore be omitted. Lastly, we aim to derive conditions that depend only on the true solution $\bx^*$, and not on the oracle $\bx^{\mathrm{or}}$. As such, $\cX$ should depend only on $\bx^*$. 

The remainder of this section will be divided into two parts. First, we will show that, under certain conditions, it is possible to ensure that the oracle solution $\bx^\mathrm{or}$, which is a local minimizer of $J_0$ by definition, is also a local minimizer of the exact relaxation $J_\Psi$.
Such result will be presented in~\cref{thm:x'locmin}. Next, we will analyze whether these conditions can be completed in order to show that $\bx^\mathrm{or}$ is the global minimizer of $J_\Psi$ and is also isolated in terms of sparsity. We present this result in~\cref{thm:oracuniquesparsest}.

\subsection{\texorpdfstring{The oracle solution is a local minimizer of $J_\Psi$}{The oracle solution is a local minimizer of JPsi}}
As already anticipated, we investigate in this section the geometrical properties of the oracle solution $\bx^\mathrm{or}$. 
To that end, we first introduce the notion of \emph{safe oracle region}, borrowing the terminology from the safe screening literature~\cite{ElGhaoui2012}.
\begin{definition}[Safe oracle region]\label{def:safe_oracle_region}
    A subset $S \subseteq \cC^{k^*}$ is said to be a safe oracle region if it contains  $\bu^{\mathrm{or}}$. Moreover, it should be defined without knowledge of $\bu^{\mathrm{or}}$.
\end{definition}
Clearly, the restriction of $\cX$ (defined in~\cref{ass:BRSCor}) to $\sigma^*$, i.e., $\cX_{\sigma^*}$, defines a safe oracle region. However, its definition involves a sublevel-set constraint that is difficult to handle in the subsequent analysis, since it couples all variables through the operator $\bA_{\sigma^*}$ in $F_{\sigma^*}$. In contrast, the proposition below introduces a safe oracle region whose construction only involves separable functions.
\begin{proposition}\label{propo:safe_region}
    Under~\cref{ass:BRSCor}, the following set 
    \begin{equation}\label{eq:setS2}
    \mathfrak{S}:=\{\bu\in\cC^{k^*} \mid F_{\sigma^*}(\bu^*)  \geq C_K D_{\Psi_{\sigma^*}}(\bu^*,\bu) \}, 
    \end{equation}
    where $D_{\Psi_{\sigma^*}}( \bu^*,\bu) := \sum_{j = 1}^{k^*} d_{\psi_{\sigma^*[j]}}(u^*_j,u_j)$, is a safe oracle region. 
\end{proposition}
\begin{proof}
Using~\cref{prop:implications_BRSC}~\cref{eq:brsc_implies_con_ineq} with $\bu^\mathrm{or}$ and $\bu^*$, we get that
$$
F_{\sigma^*}(\bu^*) - F_{\sigma^*}(\bu^\mathrm{or}) \geq \left< \nabla F_{\sigma^*}(\bu^\mathrm{or}),\bu^* - \bu^\mathrm{or}\right> + C_K D_{\Psi_{\sigma^*}}(\bu^*,\bu^\mathrm{or}).
$$
Then, since $0 \leq  F_{\sigma^*}(\bu^\mathrm{or}) \leq F_{\sigma^*}(\bu^*)$ and $\left< \nabla F_{\sigma^*} (\bu^\mathrm{or}), \bu^* - \bu^\mathrm{or}\right> \geq \mathbf{0}$ by \cref{eq:gradFzeor_xor}, we get 
$$
D_{\Psi_{\sigma^*}}(\bu^*,\bu^\mathrm{or})\leq \frac{F_{\sigma^*}(\bu^*)}{C_K},
$$
showing that $\bu^\mathrm{or}\in\mathfrak{S}$, as desired.
\end{proof}
In~\Cref{sec:Psil2}, we will see that, for the choice $\Psi=(1/2)\|\cdot\|_2^2$, this set reduces to an Euclidean ball (\cref{propo:safe_region-l2}). We are now ready to provide sufficient conditions under which $\bx^{\mathrm{or}}$ is a local minimizer of $J_\Psi$.

\begin{theorem}\label{thm:x'locmin} Let~\cref{ass:BRSCor} be satisfied and let $S \subseteq \cC^{k^*}$ be a safe oracle region. Define  $\Omega:=\bigcup_{j=1}^{k^*} \Omega_j$ where, for every $j=1,\ldots,k^*$,
$$
\Omega_j:=\left\{\bu\in\cC^{k^*} \mid u_j\in [-\alpha_{\sigma^*[j]},\alpha_{\sigma^*[j]}] \right\}.
$$
Then, under the conditions,
\begin{align}
  &   S\cap\Omega=\emptyset, \label{eq:onsup} \\
  &     \|\ba_i\|_2\sqrt{2\tilde L F(\bx^*)} \leq \psi_i'(\alpha_i) - \psi_i'(0),  \text{ for all } i\in{\sigma^*}^c,  \label{eq:offsup}
\end{align}
for some $\tilde{L} \geq L$ (with $\tilde{L} = L$ valid when $\dom G_\by = \R^M$), we have that $\bx^\mathrm{or}$ is a local minimizer of $J_\Psi$ and $\sigma(\bx^\mathrm{or})=\sigma^*$. 
\end{theorem}

We present below a sketch of the proof, with full details deferred to  ~\Cref{sec:proofs33}.

\begin{proof}[Sketch of the proof]
From \cite[Proposition 10]{Essafri2024}, we get that a local minimizer $\bx$ of $J_0$ is preserved by $J_\Psi$ if and only if
\begin{enumerate}[label=(\roman*),leftmargin=*]
\item \label{cond:suppcomps} for any $i\in \sigma(\bx)$, $|x_i|>\alpha_i$,
\item \label{cond:offsuppcomps} for any $i\in{\sigma}(\bx)^c$, $-\langle \ba_i, \nabla G_\by(\bA\bx)\rangle\in [\ell_i^-, \ell_i^+]$.
\end{enumerate}
where $\ell_i^+:=\psi'_i(\alpha_i)-\psi'_i(0)$ and $\ell_i^-:=-\psi'_i(\alpha_i)-\psi'_i(0)$ when $\cC = \R$ or $\ell_i^-=-\infty$ when $\cC=\R_{\geq 0}$.  Hence, a simple geometrical sufficient condition for $\bx^{\mathrm{or}}$ to satisfy~\cref{cond:suppcomps} above is that $S \cap \Omega = \emptyset$. Regarding~\cref{cond:offsuppcomps}, the proof consists in first showing that the term $|\langle \ba_i, \nabla G_\by(\bA\bx^{\mathrm{or}})\rangle|$ is upper bounded by $\|\ba_i\|_2(2\tilde L F(\bx^*))^{1/2}$, and combine this with~\cref{eq:offsup} to ensure that~\cref{cond:offsuppcomps} holds true.
\end{proof}

Next, notice that, while Condition~\cref{eq:offsup} is fairly explicit, Condition~\cref{eq:onsup} is more abstract and must be further specified on a case-by-case basis, depending on the choice of $G_\by$, $\Psi$, and the safe oracle region $S$ (cf.~\Cref{sec:applications}). Notice also that the above result has been presented in terms of a general safe oracle region $S$, rather than with the valid example provided in~\cref{propo:safe_region} since, clearly, other examples satisfying~\cref{def:safe_oracle_region} may be found. Finally, it is worth noting that the smaller the safe oracle region, the better the bound associated to Condition~\cref{eq:onsup}. 

\subsection{\texorpdfstring{The oracle solution is the unique sparsest global minimizer of $J_\Psi$}{unique sparse global mininizer of JPsi}} 

The present section will be devoted to analyze whether the oracle solution is a global minimizer of $J_\Psi$ (and so of $J_0$), and is also isolated in terms of sparsity for $J_\Psi$ in the sense of~\cref{thm:sparsest}. We present this result in the following theorem.

\begin{theorem}\label{thm:oracuniquesparsest}
Let~\cref{ass:BRSCor} be satisfied and let $S \subseteq \cC^{k^*}$ be a safe oracle region. Define $\Omega$ as in~\cref{thm:x'locmin} and  $\tilde{\Omega}:=\bigcup_{j=1}^{k^*} \tilde{\Omega}_{j}$  where, for every $j=1,\ldots,k^*$,    
$$
\tilde{\Omega}_j:= \left\{\bu\in\cC^{k^*} \mid \psi'_{\sigma^*[j]}(u_{j}) \in\left[-\rho_j,\rho_j\right]\right\},
$$
with $\rho_j:=(\psi_{\sigma^*[j]}'(\alpha_{\sigma^*[j]})-\psi'_{\sigma^*[j]}(0))/C_K+\psi_{\sigma^*[j]}'(0)$.

Then, under the conditions~\cref{eq:onsup} and~\cref{eq:offsup} of~\cref{thm:x'locmin} (when $C_K >1$), or the following two (stronger) conditions (when $C_K \leq 1$)
    \begin{align}
       & S\cap \tilde{\Omega}=\emptyset, \label{ass1genthm} \\
       & \|\ba_i\|_2\sqrt{2\tilde L F(\bx^*)} < C_K(\psi'_i(\alpha_i) - \psi'_i(0)), \text{ for all } i\in{\sigma^*}^c,   \label{ass2genthm}
    \end{align}
   for some $\tilde{L} \geq L$ (with $\tilde{L} = L$ valid when $\dom G_\by = \R^M$), in addition to,
   \begin{equation}\label{ass3genthm}
       \lambda_0>\frac{F(\bx^*)}{1+K-2k^*},
   \end{equation}
we have that $\bx^\mathrm{or}$ is the unique global minimizer of $J_\Psi$ (and so of $J_0$) with $\sigma(\bx^\mathrm{or})=\sigma^*$, and any other critical point $\bx'$ of $J_\Psi$, $\bx'\neq\bx^\mathrm{or}$, satisfies $\|\bx'-\bx^\mathrm{or}\|_0>K$.
\end{theorem}

As we did for~\cref{thm:x'locmin}, we now provide a sketch of the proof, and we dedicate~\Cref{proof:oracuniquesparsest} to give further details.

\begin{proof}[Sketch of the proof] The result follows by combining~\cref{thm:uniqueglob} and~\cref{thm:x'locmin}. Specifically, we first show that we are under the conditions of~\cref{thm:x'locmin}, so that $\bx^{\mathrm{or}}$ is a local minimizer (and so a critical point) of $J_\Psi$ with $\sigma(\bx^\mathrm{or})=\sigma^*$, and second, we show that it is possible to apply~\cref{thm:uniqueglob} to $\bx^{\mathrm{or}}$, proving that $\bx^{\mathrm{or}}$ is the unique global minimizer of $J_\Psi$ (and so of $J_0$), and is also isolated in terms of sparsity. Notice that the required BRSC property at $\bx^{\mathrm{or}}$ is satisfied in this case thanks to~\cref{ass:BRSCor}. 

To do so, and as already indicated in the result, we will need to distinguish two cases: when $C_K>1$, and when $C_K\leq 1$. The first case, $C_K>1$, is trivial, as here we precisely require the conditions~\cref{eq:onsup} and~\cref{eq:offsup} to hold true, and so,~\cref{thm:x'locmin} directly follows. Additionally, Condition~\cref{eq:zinotin}  reduces to $|z_i| \notin \emptyset$ and is thus trivially satisfied. Finally, observing that
\begin{equation}
    \lambda_0 > \frac{F(\bx^*)}{1+K-2k^*} \geq \frac{F(\bx^{\mathrm{or}})}{1+K-2k^*}, 
\end{equation}
where in the last inequality we simply used the definition of $\bx^\mathrm{or}$, completes the requirements of~\cref{thm:uniqueglob}, and proves this first part of the result. On the other hand, in the more involved case $C_K\leq 1$, we will prove that conditions~\cref{ass1genthm} and~\cref{ass2genthm} imply, on the one hand, those of~\cref{thm:x'locmin} (i.e., conditions~\cref{eq:onsup} and~\cref{eq:offsup}) and, on the other hand, Condition~\cref{eq:zinotin}. By repeating the same argument as in the case above, we can conclude.
\end{proof}

Similarly to~\cref{thm:x'locmin}, the conditions stated in~\cref{thm:oracuniquesparsest} must be further simplified on a case-by-case basis, depending on the specific choice of $G_\by$, $\Psi$, and~$S$. This is the purpose of the next section, where we particularize~\cref{thm:oracuniquesparsest} to different data fidelity terms, namely the least-squares loss $G_\by^{\mathrm{LS}}$ and the generalized Kullback–Leibler divergence $G_\by^{\mathrm{KL}}$. In particular, we anticipate that, in these specific cases, the conditions of~\cref{thm:oracuniquesparsest} will take the following form: 
\begin{equation}\label{eq:sqrtlambinterval}
\lambda_0 \in \left(\underline{\Lambda}(\by,C_K,L), \bar{\Lambda}(\bx^*,\by,C_K,L)\right).
\end{equation}
While~\cref{ass3genthm} involves $\lambda_0$ explicitly, Conditions~\cref{ass1genthm} and~\cref{ass2genthm} (or~\cref{eq:onsup} and~\cref{eq:offsup}) depend on $\lambda_0$ through the $\alpha_i$'s ($d_{\psi_i}(0,\alpha_i)=\lambda_0$). 

The above form will allow us to analyze the asymptotic behavior of both $\underline\Lambda$ and $\bar\Lambda$ with respect to both the noise level and the amplitudes of the nonzero components of the true solution $\bx^*$. Before doing so, we discuss the noiseless case in the following remark, which concludes the section.

\begin{remark}[On the noiseless setting] Let us consider the extreme case in which no noise is present in the observations; i.e., when $\by=\by^*$. Additionally, we set $S=\mathfrak{S}$, the safe oracle region given in~\cref{propo:safe_region}. Now, observe that Conditions~\cref{ass1genthm},~\cref{ass2genthm} and~\cref{ass3genthm} depend on $F(\bx^*)$, either directly as in~\cref{ass2genthm} and~\cref{ass3genthm}, or indirectly as in~\cref{ass1genthm}, where the term $F_{\sigma^*}(\bu^*)=F(\bx^*)$ is involved in the definition of $\mathfrak{S}$. Hence, in order to properly characterize the noiseless regime, we must specify how $F(\bx^*)$ behaves as the noise level decreases. In particular, $F$ should be chosen so that its value reflects the mismatch between the predicted and true observations, and naturally vanishes when the two coincide. Thus, in the noiseless case we should have $F(\bx^*)= G_{\by^*}(\bA\bx^*)=0$. Notice that this is the case for both of the data fidelity terms in~\cref{ex:LS} and~\cref{ex:KL} (in the latter, observe that $G^{\mathrm{KL}}_{\by^*}(\bA\bx^*)= \sum_{j=1}^M [\bA\bx^*]_j+b_j+y_j^*\log(y_j^*/([\bA\bx^*]_j+b_j)) -y_j^*=0$ because $\by^*=\bA\bx^*+\mathbf{b}$) and, actually, for any data term that can be written as the Bregman divergence of a convex function. Taking this into account, we have that Condition~\cref{ass1genthm} writes $\{\bu^*\}\cap\tilde\Omega=\emptyset$. Next, Condition~\cref{ass2genthm} reduces to $0\leq \psi_i'(\alpha_i) - \psi_i'(0)$, which is always true from~\cref{ass:generating_func_psi}~\ref{ass3psi}. Finally, Condition~\cref{ass3genthm} reads $\lambda_0>0$. Combining everything, we conclude that, in the noiseless regime, the interval in~\cref{eq:sqrtlambinterval} becomes $(0,\bar{\Lambda})$, where $\bar{\Lambda}$ has to be deduced from $\{\bu^*\}\cap\tilde\Omega=\emptyset$.
\end{remark}

\section{Applications}\label{sec:applications}

We dedicate this section to further explore the implications of our theoretical results to both~\cref{ex:LS} and~\cref{ex:KL}. To do so, we must proceed in two steps. First, given a data fidelity term, we need to determine the Bregman function $\Psi$ such that~\eqref{eq:brsc} holds. Second, we must particularize our general bounds to the specific sparse optimization problem at hand. As we show in both~\Cref{sec:Psil2} and~\Cref{sec:exp_KL}, both tasks require significant  analytical effort. 

Furthermore, while we restrict our specialization to these two examples, our theoretical framework is designed to accommodate a much broader class of sparse optimization problems. More generally, any data fidelity term that is the Bregman divergence of a strictly convex function naturally falls within the setting of~\cref{ass:FID}. For example, the data fidelity terms described in~\cref{ex:LS} and~\cref{ex:KL} are particular instances ($\beta=2$ and $\beta=1$, respectively) of the broader class of $\beta$-divergences~\cite{Basu1998,Cichocki2010}. We expect that our analysis extends, for instance, to the critical value $\beta=1.5$, a case that proves useful in applications such as hyperspectral imaging~\cite{FD2015}. Similarly, robust formulations like the Cauchy regression problem, commonly deployed to handle heavy-tailed impulsive noise or extreme outliers~\cite{Carr2013,Guo2023}, should also fit within our paradigm. In conclusion, our theory can be applied to a larger variety of problems beyond those presented here, albeit requiring some non trivial work for each specific case.

As in the above sections, we set $F:=G_\by(\bA\cdot)$. First, in~\Cref{sec:Psil2} we set the Bregman-generating function $\Psi$ to be the squared $\ell_2$-norm and particularize~\cref{thm:oracuniquesparsest} in this setting. Additionally, we show that, if $F=G^{\mathrm{LS}}_\by(\bA\cdot)$, our analysis brings improved results with respect to existing literature~\cite{Carlsson2020}. Second, in~\Cref{sec:exp_KL}, we set $F=G^{\mathrm{KL}}_\by(\bA\cdot)$ and~$\Psi$ to a suitable Kullback--Leibler Bregman-generating function. From~\cref{propo:positiveKL}, it is rather direct to see that~\cref{ass:BRSCor} holds under an LRIP condition on the matrix $\bA$ and, consequently, a corresponding version of~\cref{thm:oracuniquesparsest} can be provided. 

\subsection{\texorpdfstring{Quadratic generating function ($\Psi=\ell_2$)}{Psi=l2 and Lipschitz gradient data terms}}\label{sec:Psil2}

In this section, we aim to rewrite~\cref{thm:oracuniquesparsest} for the particular choice 
\begin{equation}\label{eq:Psi_l2}
    \Psi(\bx)=\frac{\gamma}{2} \|\bx\|_2^2
\end{equation}
where the scalar $\gamma> 0$ is set according to the concavity condition~\cref{eq:CC}; that is
\begin{equation}\label{eq:CC_psi_l2}
    \gamma > \max_{i=1,\ldots,N} \sum_{j=1}^M a_{ji}^2\sup_{w \in \cC} g_{y_j}''(w).
\end{equation}
Note that we consider a scalar $\gamma$ to simplify the presentation. Yet, all the results of this section can be easily adapted if we instead consider a vector $\boldsymbol{\gamma} \in \R_{>0}^N$, weighting differently each component of $\bx$ as in~\cite{Essafri2024}. For further convenience, a $\Psi$ chosen as in~\cref{eq:Psi_l2} with $\gamma$ according to \cref{eq:CC_psi_l2} will be denoted as $\Psi=\ell_2$.

We next provide some insights about~\cref{ass:BRSCor}. Notice that, in this case ($\Psi=\ell_2$), it turns into requiring $F$ to satisfy~\cref{eq:RSC} within a certain compact set $\cX$ containing both $\bx^{\mathrm{or}}$ and $\bx^*$. Hence, from~\cref{propo:BRSC_with_psiL2}, we get that~\cref{eq:LRIP} is a sufficient condition for~\cref{eq:RSC} to hold true, whenever $G_\by$ is $\nu-$strongly convex, with constant $\mu_K=\nu(1-\delta_K^-)/\gamma$. In that case~\cref{ass:BRSCor} holds true with $C_K = \mu_K$. Next, we explicit the safe oracle region described in~\cref{propo:safe_region} to this context. Interestingly, it turns to be an Euclidean ball.

\begin{proposition}\label{propo:safe_region-l2}
    For $\Psi=\ell_2$, the safe oracle region $\mathfrak{S}$ given in~\cref{propo:safe_region} reduces to the Euclidean ball
    $$
    \mathfrak{S}= B \left(\bu^* , \sqrt{\frac{2  F_{\sigma^*}(\bu^*)}{\gamma C_K }} \right).
    $$
\end{proposition}
\begin{proof}
If $\bu\in\mathfrak{S}$, we have that 
$$
    F_{\sigma^*}(\bu^*) \geq \frac{\gamma C_K}{2} \|\bu - \bu^*\|_2^2 \; \Longrightarrow \; \bu \in B \left(\bu^* , \sqrt{\frac{2  F_{\sigma^*}(\bu^*)}{\gamma C_K}} \right)
$$
which completes the proof.
\end{proof}
We now present the corresponding version of~\cref{thm:oracuniquesparsest} when $\Psi = \ell_2$.

\begin{corollary}\label{coro:general_result_l2}
    Let $\Psi=\ell_2$,~\cref{ass:BRSCor} be satisfied, and consider the safe oracle region $\mathfrak{S}$ given in~\cref{propo:safe_region-l2}. Define
    \begin{align}
        \underline \Lambda := F(\bx^*) \max \left\lbrace \frac{\tilde L}{\gamma\min\{C_K^2,1\}}  \max_{i \in \sigma^{*c}} \| \ba_i\|_2^2  ,  \frac{1}{1+K-2k^*} \right\rbrace,
    \end{align}
     as well as
    \begin{align}\label{eq:coro:general_result_l2bar} 
        \bar{\Lambda} := \frac{\gamma }{2} \min\{C_K^2,1\} \left( \min_{i \in \sigma^*} |x_i^*| - \sqrt{\frac{2 F(\bx^*)}{\gamma C_K} } \right)^2.
    \end{align}
    Then, if $\lambda_0 \in (\underline \Lambda, \bar \Lambda)$, we have that $\bx^\mathrm{or}$ is the unique global minimizer of $J_\Psi$ (and so of $J_0$) with $\sigma(\bx^\mathrm{or})=\sigma^*$, and any other critical point $\bx'$ of $J_\Psi$, $\bx'\neq\bx^\mathrm{or}$, satisfies $\|\bx'-\bx^\mathrm{or}\|_0>K$.
\end{corollary}

We dedicate~\Cref{proof:corolgenL2} to prove the above result. 

\subsubsection{Gaussian regression}\label{sec:quadfit}  
Here, we consider the setting described in~\cref{ex:LS}. Additionally,  we assume in this case, without loss of generality, that
\begin{equation}\label{eq:cond_norm_col_A}
    \|\bA\|_{1,2} := \max_{i=1,\ldots,N} \|\ba_i\|_2 < 1.
\end{equation}
By recalling that $\bA$ is a linear operator between finite dimensional spaces, hence bounded, we point out that the above requirement can always be satisfied through rescaling. In this context, the concavity condition \cref{eq:CC} turns into requiring $\gamma > \|\bA\|_{1,2}^2$ (see~\cite[Table 5]{Essafri2024}). Recalling~\cref{eq:cond_norm_col_A}, we may thus fix $\gamma=1$. Additionally, let us assume that~\cref{eq:LRIP} holds with constant $\delta_K^-<1$. Now, as already mentioned,~\cref{propo:BRSC_with_psiL2} implies that~\cref{ass:BRSCor} holds true with constant $C_K=\nu(1-\delta_K^-)/\gamma=\nu(1-\delta_K^-) = 1-\delta_K^-$ (notice that $G^{\mathrm{LS}}_\by$ is strongly convex with parameter $\nu=1$). In order to apply~\cref{coro:general_result_l2}, it remains to further specify the constants defined therein.

\begin{proposition} \label{propo:l2}Under the least-squares setting described above, the quantities of~\cref{coro:general_result_l2} are such that
\begin{align*}
    & \underline{\Lambda} <  \frac{\|\beps\|_2^2}{2\min \{(1-\delta_K^-)^2,1\}},  \\
    & \bar \Lambda =   \frac{\min \{(1-\delta_K^-)^2,1\}}{2}\left( \min_{i \in \sigma^*} |x_i^*| - \frac{\|\beps\|_2}{\sqrt{1-\delta_K^-}} \right)^2.
\end{align*}
\end{proposition}
\begin{proof}
    Let $F=G_\by^{\mathrm{LS}}(\bA\cdot)$. Then, $F(\bx^*) = (1/2)\|\beps\|^2_2$ and,
    since $C_K = 1-\delta_K^-$, we have that $\min\{C_K^2,1\} = \min\{(1-\delta_K^-)^2,1\}$. Additionally, recall that in this case we have $\dom G^{\mathrm{LS}}_\by=\R^M$, and so, we may choose $\tilde L=L=1$. Injecting these expressions in~\cref{eq:coro:general_result_l2bar} lead to the claimed upper bound $\bar \Lambda$. Regarding the lower bound we have
    \begin{align*}
        & \underline{\Lambda} < \frac{\|\beps\|_2^2}{2} \max \left\lbrace \frac{1}{\min \{(1-\delta_K^-)^2,1\}}, \frac{1}{1+K-2k^*}  \right\rbrace = \frac{\|\beps\|_2^2}{2\min \{(1-\delta_K^-)^2,1\}}
    \end{align*}
    where the strict inequality comes from the assumption that $\max_{i \in \sigma^{*c}} \| \ba_i\|_2^2 <1$ and the last equality is due to the fact that $K \geq 2k^*$ by assumption.
\end{proof}

To interpret this result, notice that, combining~\cref{propo:l2} and~\cref{coro:general_result_l2}, the oracle solution $\bx^{\mathrm{or}}$ is the global minimizer of both $J_{\Psi}$ and $J_0$, and is in addition isolated in terms of sparsity for $J_\Psi$, if
\begin{equation}\label{eq:lamb_0_interval_Psi_l2}
    \lambda_0 \in \left( \frac{\|\beps\|_2^2}{2\min \{(1-\delta_K^-)^2,1\}} , \frac{\min \{(1-\delta_K^-)^2,1\}}{2}\left( \min_{i \in \sigma^*} |x_i^*| - \frac{\|\beps\|_2}{\sqrt{1-\delta_K^-}}\right)^2 \right).
\end{equation}
Hence, provided that the noise level $\|\beps\|_2$ is sufficiently small compared to the nonzero components of the true solution $\bx^*$, there exists a nonempty range of regularization parameters $\lambda_0$ such that the result holds true. 

To conclude with the least-squares problem, we add a final remark showing that our analysis allows us to improve existing results in the literature.

\begin{remark} [Comparison with the bounds derived in~\cite{Carlsson2020}] It was previously proved in~\cite[Theorem 4.9]{Carlsson2020} that the claimed result holds true (for the case $\delta_K^- \geq 0$ treated in~\cite{Carlsson2020}) if
\begin{equation}\label{eq:interval_Marcus}
    \lambda_0 \in \left(\frac{\|\beps\|_2^2}{2(1-\delta_K^-)^2}, \frac{(1-\delta_K^-)^2}{2(2 - \delta_K^-)^2}   \min_{i \in \sigma^*} |x_i^*|^2 \right).
\end{equation}
Note that the $2$ factors are not present in~\cite{Carlsson2020} as the least-squares loss is not weighted with $(1/2)$. This makes that each occurrence of the regularization parameter (denoted $\mu$ in~\cite{Carlsson2020}) has to be replaced by $2\lambda_0$ in our setting.

As it can be seen, we recovered exactly the lower bound of~\cref{eq:interval_Marcus} in~\cref{propo:l2}. Regarding the upper bound, let $\lambda_0$ satisfy~\cref{eq:interval_Marcus}. Then, as $(1-\delta_K^-) / (1-\delta_K^-)^{1/2} \leq 1$ for $\delta_K^- \geq 0$, we have that
$$
    \min_{i \in \sigma^*} |x_i^*|^2  > \frac{(2-\delta_K^-)^2}{(1-\delta_K^-)^2} {2\lambda_0} = \left( \frac{1}{1 - \delta_K^-} + 1 \right)^22\lambda_0 \geq \left( \frac{1}{1 - \delta_K^-} + \frac{1-\delta_K^-}{\sqrt{1-\delta_K^-}} \right)^2{2\lambda_0} 
$$
and, using that $1 - \delta_K^- \geq \|\beps\|_2 / \sqrt{2\lambda_0}$, we see that
\begin{align*}
     \min_{i \in \sigma^*} |x_i^*|^2> \left( \frac{1}{1-\delta_K^-} + \frac{\|\beps\|_2}{\sqrt{1-\delta_K^-}\sqrt{2\lambda_0}} \right)^2{2\lambda_0} 
\end{align*}
which implies that ${\lambda_0}$  belongs to the interval defined by~\cref{propo:l2} and~\cref{coro:general_result_l2}.  Hence, our upper bound is better (i.e. larger) than that of~\cref{eq:interval_Marcus}.
\end{remark}

\subsection{Poisson regression}\label{sec:exp_KL}

We now explore~\cref{ex:KL}, where $F=G_\by^{\mathrm{KL}}(\bA\cdot)$. Addressing the optimization landscape of this specific data fidelity term is of major practical importance, as it is the canonical choice for handling count data and signal-dependent noise in a wide variety of inverse problems, including Poisson intensity reconstruction and non-negative matrix factorization~\cite{Figue2010,FI2011,Harmany2012}, and image reconstruction tasks~\cite{hurault2023,klatzer2025,Daniele2025}. To apply our theoretical framework in this setting, the choice of the Bregman-generating function $\Psi$ proves critical. Indeed, if one simply sets $\Psi=\ell_2$, and as a direct consequence of~\cref{propo:negativeKL}, it turns out that~\cref{ass:BRSCor} does not hold in this setting, and so,~\cref{thm:oracuniquesparsest} cannot be applied. However, we have shown in~\cref{propo:positiveKL} that choosing the (smoothed) Burg entropy as the Bregman-generating function yields positive BRSC constants. Motivated by this result, and inspired by~\cite{EssafriKL} to ensure the exact relaxation property to $J_\Psi$, we consider the following KL Bregman-generating function: for every $\bx\in\R^N_{\geq 0}$, we fix
\begin{equation}\label{eq:PsiKL}
\Psi(\bx)=\sum_{i=1}^N \psi_i(x_i), \quad \text{with } 
\psi_i(x) := \gamma_i g^{\mathrm{KL}}_\xi(c_ix+\xi)
\end{equation}
where, we recall, $g^{\mathrm{KL}}_\xi (z) = z + \xi \log(\xi /z) - \xi$, and the parameters $\xi>0$, $\gamma_i, c_i>0$, $i=1,\ldots, N$, are set according to the concavity condition~\cref{eq:CC}, that is~\cite{EssafriKL}:
\begin{equation}\label{eq:CC_for_KL}
c_i=\min_{j\in\sigma(\ba_i)}a_{ji}, \quad \gamma_i>\sum_{j=1}^M\frac{a_{ji}^2y_j}{c_i^2 \xi},\quad \text{and } \xi\leq \min_{j=1,\ldots, M} b_j.
\end{equation}
Notice that, in the above, $c_i>0$ for all $i=1,\ldots, N$ since the minimum is taken among components within $\sigma(\ba_i)$, and $\sigma(\ba_i)\neq \emptyset$ for all $i=1,\ldots, N$ because, by assumption, the matrix $\bA$ does not have any column of only zeros. Moreover, one can easily check that such $\Psi$ satisfies each of the conditions given in~\cref{ass:generating_func_psi}. Note that the above choice of $\Psi$ slightly differs from the one proposed in~\cite{EssafriKL}, in the sense that we use $g^{\mathrm{KL}}_\xi$ instead of $g^{\mathrm{KL}}_1$ in~\cite{EssafriKL}. This change was of importance to get relevant bounds in~\cref{coro:KL_application} below for any $\mathbf{b} \in \R^M_{>0}$. For simplicity, such $\Psi$ satisfying the concavity condition described above will be denoted as $\Psi=\mathrm{KL}$. Now, we will show that ~\cref{ass:BRSCor} holds. To do so, let us denote the Burg entropy~\cref{eq:burg} as $\Phi$ with $\eta_i:=\xi/c_i$ for all $i=1,\ldots, N$ and $\Psi=\mathrm{KL}$ forming B-rex as described above. Observe that, for all $\bx, \bx'\in\cC^N$ we have
$$
\begin{aligned}
D_{\Psi}^{\ \mathrm{symm}}(\bx, \bx')&=\sum_{i=1}^N \gamma_i \frac{\xi c_i^2(x_i-x_i')^2}{(c_ix_i + \xi)(c_i x_i' + \xi )}\\
&=\sum_{i=1}^N \gamma_i \frac{\xi(x_i-x_i')^2}{(x_i + \xi/c_i)(x_i' + \xi/c_i )}\\
& \leq \xi \left(\max_i \gamma_i \right)D^{\ \mathrm{symm}}_{\Phi}(\bx, \bx').
\end{aligned}
$$
Then, since $F$ satisfies~\cref{eq:brsc} with respect to $\Phi$ with constant $C_K>0$ by~\cref{propo:positiveKL}, we conclude that $F$ satisfies~\cref{eq:brsc} with respect to $\Psi=\mathrm{KL}$ with constant $\tilde{C}_K=C_K/[\xi \left(\max_i \gamma_i \right)]$. 

Next, we detail the safe oracle region $\mathfrak{S}$ from~\cref{propo:safe_region}.
\begin{proposition}\label{propo:oracregKL} With $\Psi=\mathrm{KL}$, the safe oracle region $\mathfrak{S}$ in~\cref{propo:safe_region} becomes
$$
\mathfrak{S}=\left\{\bu\in\cC^{k^*} \mid \sum_{j=1}^{k^*}\gamma_{\sigma^*[j]}g^{\mathrm{KL}}_1\left(\frac{c_{\sigma^*[j]}u_j^*+\xi}{c_{\sigma^*[j]}u_j+\xi}\right)\leq \frac{F_{\sigma^*}(\bu^*)}{\xi C_K} \right\}
$$
\end{proposition}
\begin{proof} Recall that $\bu\in\mathfrak{S}$ means
$$
D_{\Psi_{\sigma^*}}(\bu^*,\bu)=\sum_{j=1}^{k^*}\gamma_{\sigma^*[j]}\left(\psi_{\sigma^*[j]}(u_j^*)-\psi_{\sigma^*[j]}(u_j)-\psi'_{\sigma^*[j]}(u_j)(u_j^*-u_j)\right)\leq\frac{F_{\sigma^*}(\bu^*)}{C_K}.
$$
Using the specific form of $\Psi=\mathrm{KL}$, the above is equivalent to
$$
\sum_{j=1}^{k^*}\gamma_{\sigma^*[j]}\left(\log\left(\frac{c_{\sigma^*[j]}u_j+\xi}{c_{\sigma^*[j]}u^*_j+\xi}\right)+\frac{c_{\sigma^*[j]}(u_j^*-u_j)}{c_{\sigma^*[j]}u_j+\xi}\right)\leq\frac{F_{\sigma^*}(\bu^*)}{\xi C_K}
$$
where, up to the multiplicative terms $\gamma_{\sigma^*[j]}$, the elements of the sum in the left hand side satisfy, for every $j=1\ldots, k^*$,
$$
\begin{aligned}
\log\left(\frac{c_{\sigma^*[j]}u_j+\xi}{c_{\sigma^*[j]}u^*_j+\xi}\right)+\frac{c_{\sigma^*[j]}(u_j^*-u_j)}{c_{\sigma^*[j]}u_j+\xi} &= \frac{c_{\sigma^*[j]}u_j^*+\xi}{c_{\sigma^*[j]}u_j+\xi}-\log\left(\frac{c_{\sigma^*[j]}u^*_j+\xi}{c_{\sigma^*[j]}u_j+\xi}\right)-1\\
&=g^{\mathrm{KL}}_1\left(\frac{c_{\sigma^*[j]}u_j^*+\xi}{c_{\sigma^*[j]}u_j+\xi}\right),
\end{aligned}
$$
concluding.
\end{proof}

In order to write the corresponding version of~\cref{thm:oracuniquesparsest}, and ease its interpretation, we let $\beps := \by - \by^* = \by - \bA\bx^* - \mathbf{b}$. Then, as $y_j^* = [\bA \bx^*]_j + b_j$, we have
\begin{equation}\label{eq:up_bound_Fxstar}
    F(\bx^*) = \sum_{j=1}^M \left[(y^*_j + \varepsilon_j)\log\left(\frac{y^*_j+\varepsilon_j}{y_j^*}\right) -  \varepsilon_j\right] \leq f(\|\beps\|_\infty, m(\bx^*)), 
\end{equation}
where $m(\bx^*) = \min_{i \in \sigma^*} x^*_i$, and the function $f$ is defined, for all $\varepsilon, x\in\R_{\geq 0}$, as
$$
f(\varepsilon,x) := \sum_{j=1}^M \left( (a_j x + b_j + \varepsilon)\log \left( \frac{a_jx + b_j+ \varepsilon}{a_jx + b_j}\right)   - \varepsilon \right),
$$
being $a_j : =  \sum_{i \in \sigma^*} a_{j,i}$. Next, notice that $f$ is decreasing whenever $\varepsilon\to 0$ or $x\to \infty$. Additionally, 
\begin{equation}\label{eq:limitsfepsx}
\lim_{x\to+\infty}f(\varepsilon, x)=0,\quad \text{ and } \lim_{\varepsilon\to 0}f(\varepsilon, x)=0.
\end{equation}
Combining everything, we see that  $F(\bx^*) = F_{\sigma^*}(\bu^*)$ decreases when either the noise level $\|\beps\|_\infty$ decreases or the nonzero elements of $\bx^*$ increase. We are now ready to particularize~\cref{thm:oracuniquesparsest} to this context.

\begin{corollary} \label{coro:KL_application} Let $\Psi=\mathrm{KL}$, let~\cref{ass:BRSCor} be satisfied, and consider the safe oracle region $\mathfrak{S}$ given by \cref{propo:oracregKL}.
Moreover, define
\begin{equation}\label{eq:lowerlambKL}
 \mkern-5mu \underline \Lambda  \mkern-4mu   := \mkern-4mu \max \mkern-3mu \Bigg\{\frac{f(\|\beps\|_\infty, m(\bx^*))}{1+K-2k^*}, \max_{i\in{\sigma^*}^c}\Big\{- \gamma_i \xi \log \Big(1  - \frac{\|\ba_i\|_2 \sqrt{2\tilde Lf(\|\beps\|_\infty, m(\bx^*))} }{\min\{C_K,1\} \gamma_i c_i} \Big)\Big\}\Bigg\}  \mkern-10mu
\end{equation}
as well as
\begin{equation}\label{eq:upperlambKL}
\bar \Lambda :=\min_{i\in\sigma^*}\left\{-\gamma_i\xi\left(\log\left(1-e^{\displaystyle h_i(\|\beps\|_\infty, m(\bx^*))}\right)+1\right)\right\},
\end{equation}
for some functions (see proof) $h_i\colon \R_{\geq 0}\times\R_{\geq 0}\to \R_{\leq 0}$ (which differ depending if $C_K < 1$ or $C_K \geq1$), $i\in\sigma^*$, that are increasing when $\|\beps\|_\infty\to 0$ or $m(\bx^*)\to+\infty$. 

Then, if $\lambda_0 \in (\underline \Lambda, \bar \Lambda)$, we have that $\bx^\mathrm{or}$ is the unique global minimizer of $J_\Psi$ (and so of $J_0$) with $\sigma(\bx^\mathrm{or})=\sigma^*$, and any other critical point $\bx'$ of $J_\Psi$, $\bx'\neq\bx^\mathrm{or}$, satisfies $\|\bx'-\bx^\mathrm{or}\|_0>K$.

Additionally, the $h_i$'s increases towards
\begin{enumerate}[label=(\roman*),leftmargin=*]
    \item $(c_i u_i^*(\min\{C_K,1\}-1)-\xi)/(c_iu_i^* + \xi)$ when $\|\beps\|_\infty \to 0$,
    \item $\min\{C_K,1\} -1$ when $m(\bx^*) \to \infty$.
\end{enumerate}
\end{corollary}

We dedicate~\Cref{proofcorolKL} to prove the result. 


In order to provide an understandable interpretation of the above result, we distinguish both of the estimates; i.e., we analyze the behavior of both $\underline \Lambda$ and $\bar \Lambda$ with respect to $\|\beps\|_\infty$ and $m(\bx^*)$. Notice that, given the variations and limits of $f$ and $h_i$, $i \in \sigma^*$, if either the noise level tends to zero or the nonzero components of $\bx^*$ tend to $+\infty$, one observes that $1)$ $\underline \Lambda$ decreases towards $0$, and $2)$  $\bar \Lambda$ increases. In particular, if $C_K\geq 1$ and $m(\bx^*)\to+\infty$, we have that $\bar \Lambda$ increases towards $+\infty$. In all, we see that the interval $(\underline \Lambda,\bar \Lambda)$ enlarges whenever the noise level $\|\beps\|_\infty$ tends to zero and/or the nonzero components of the true solution $\bx^*$ tend to $+\infty$, therefore indicating that, as long as the noise level is small enough compared to the nonzero components of $\bx^*$, the interval $(\underline \Lambda, \bar \Lambda)$ is nonempty.

\section{Proofs}

\subsection{\texorpdfstring{Proofs of~\Cref{sec:brsc}}{Proos of section 2}}

\subsubsection{\texorpdfstring{Proof of~\cref{propo:BRSC_with_psiL2}}{Proof of Proposition 2.1}}\label{proof:propoBRSC_Psil2}

To prove the result, we need to show that \cref{eq:RSC} holds for all $\bx, \bx'\in\cC^N$ with $\|\bx-\bx'\|_0\leq K$ and with constant $\mu_K=(1-\delta_K^-)\nu$. To do so, it suffices to prove that, for all $\omega\subseteq\{1,\ldots, N\}$ with $\#\omega\leq K$, we have
$$
D^{\ \mathrm{symm}}_{G(\bA_\omega\cdot)}(\bx_\omega, \bx_\omega')\geq (1-\delta_K^-)\nu \|\bx_\omega -\bx_\omega'\|_2^2
$$
for all $\bx, \bx'\in\cC^N$ with $\bx_{\omega^c} = \bx'_{\omega^c}$ and $\bx_{\omega} \neq \bx'_{\omega}$. Let $\omega$ with $\#\omega\leq K$ and observe that, as $G$ is strongly convex with parameter $\nu>0$,
\begin{align*}
D^{\ \mathrm{symm}}_{G(\bA_\omega\cdot)}(\bx_\omega, \bx_\omega')&=\langle \bA_\omega^T\nabla G(\bA_\omega\bx_\omega)-\bA_\omega^T\nabla G(\bA_\omega\bx_\omega'), \bx_\omega-\bx_\omega'\rangle\\
&=\langle \nabla G(\bA_\omega\bx_\omega)-\nabla G(\bA_\omega\bx_\omega'), \bA_\omega\bx_\omega-\bA_\omega\bx_\omega'\rangle\\
&\geq \nu\|\bA_\omega\bx_\omega-\bA_\omega\bx_\omega'\|_2^2,
\end{align*}
for all $\bx, \bx'\in\cC^N$ as above. Finally, as~\cref{eq:LRIP} holds by assumption, we conclude that 
\begin{align*}
D^{\ \mathrm{symm}}_{G(\bA_\omega\cdot)}(\bx_\omega, \bx_\omega')\geq (1-\delta_K^-)\nu\|\bx_\omega-\bx_\omega'\|_2^2,
\end{align*}
as desired.

\subsubsection{\texorpdfstring{Proof of~\cref{propo:negativeKL}}{Proof of Proposition 5.8}}\label{proof:negativeKL}
To prove the result, we only need to show that $C_1 = 0$. Then, from~\cref{prop:orderBRSCconst} we directly get that $C_K \leq C_1 = 0$ for any $K\geq1$. We recall that~\cref{eq:brsc} in this case reads: for all $\bx\in\cX$, $\bx'\in\R^N$ with $\bx\neq\bx'$ and $\|\bx-\bx'\|_0\leq 1$,
\begin{equation}\label{eq:BRSC_0_LR}
\left<\bA^T(\nabla G_\by^{\mathrm{KL}}(\bA\bx) - \nabla G_\by^{\mathrm{KL}}(\bA\bx')), \bx - \bx' \right>\geq C_1 \|\bx - \bx'\|_2^2.
\end{equation}
Now, let $\omega=\{i_0\}$ with $i_0\in\{1,\ldots, N\}$, and observe that a necessary condition for~\cref{eq:BRSC_0_LR} to hold in this case is
$$
\sum_{j=1}^M a_{ji_0}\left( [\nabla G^{\mathrm{KL}}_\by(\bA \bx)]_j - [\nabla G^{\mathrm{KL}}_\by(\bA_{\omega^c} \bx_{\omega^c} + \ba_{i_0}x_{i_0}')]_j \right) (x_{i_0} - x_{i_0}') \geq  C_1 (x_{i_0} - x_{i_0}')^2
$$
for all $\bx\in\cX$ and $\bx'\in\R^N$ with $x_i=x_i'$ for $i\neq i_0$ and $x_{i_0}\neq x'_{i_0}$.  Rearranging and developing the expression of $G_\by^{\mathrm{KL}}$, we derive that $C_1$ is such that 
$$
\begin{aligned}
C_1& \leq\inf_{\substack{\bx\in\cX, \\ x'_{i_0}\in\R_{\geq 0} }} \sum_{j=1}^M \frac{a_{ji_0}\left( [\nabla G^{\mathrm{KL}}_\by(\bA \bx)]_j - [\nabla G^{\mathrm{KL}}_\by(\bA_{\omega^c} \bx_{\omega^c} + \ba_{i_0}x_{i_0}')]_j \right)}{x_{i_0} - x_{i_0}'}\\
&=\inf_{\substack{\bx\in\cX, \\ x'_{i_0}\in\R_{\geq 0} }} \sum_{j=1}^M\frac{y_ja^2_{ji_0}}{([\bA\bx]_j+b_j)(a_{ji_0}x'_{i_0}+[\bA_{\omega^c}\bx_{\omega^c}]_j+b_j)}.
\end{aligned}
$$
Next, we fix $\bx\in\cX$ and analyze each term of the above sum when $x'_{i_0}$ tends to $+\infty$. They are functions of the form
$$
h(x')=\frac{ya^2}{\delta_1ax'+\delta_2},
$$
with $\delta_1, \delta_2>0$ for $i=1, 2$. Now, if $y=0$, then $h(x')=0$ for all $x'\geq 0$, and this term does not contribute to the sum. Hence, we assume that $y>0$ and notice that, if $a=0$, then again $h(x')=0$ for all $x\geq 0$, and this term does not contribute to the sum either. Therefore, the only terms contributing to the above sum are those in which $y, a>0$. However, these terms are such that
$$
\lim_{x' \to +\infty}  \frac{ya}{\delta_1ax'+\delta_2} = 0,
$$
Combining everything, we conclude that $C_1=0$, as desired.

\subsubsection{\texorpdfstring{Proof of~\cref{propo:positiveKL}}{Proof of Theorem 5.7}}\label{proofthmKL}

We need to show that there exists $C_K>0$ such that 
\begin{equation}\label{eq:BRSC_for_KL}
\sum_{j =1}^{\#\sigma_\by} \frac{y_{\sigma_\by(j)}[\tilde{\bA}(\bx-\bx')]_j^2}{([\tilde{\bA}\bx]_j+b_{\sigma_\by(j)})([\tilde{\bA}\bx']_j+b_{\sigma_\by(j)})}\geq  C_K
\sum_{i=1}^N \frac{(x_i-x_i')^2}{(x_i + \eta_i)(x_i' + \eta_i )}
\end{equation}
for all $\bx\in\cX$ and $\bx'\in\cC^N$ with $\|\bx-\bx'\|_0\leq K$; and to do so, it suffices to show that, for all $\omega \subseteq \{1,\ldots, N\}$ with $\#\omega \leq K$,  there exists $C_\omega>0$ such that
\begin{equation}\label{eq:BRSC_for_KL-2}
\sum_{j =1}^{\#\sigma_\by} \frac{y_{\sigma_\by(j)}[\tilde{\bA}_\omega(\bx_\omega-\bx'_\omega)]_j^2}{([\tilde{\bA}\bx]_j+b_{\sigma_\by(j)})([\tilde{\bA}_\omega\bx'_\omega]_j \mkern-3mu +[\tilde{\bA}_{\omega^c}\bx_{\omega^c}]_j +b_{\sigma_\by(j)})}\geq  C_\omega \sum_{i \in \omega}  \frac{(x_i-x_i')^2}{(x_i + \eta_i)(x_i' + \eta_i)}
\end{equation}
for all $\bx\in\cX$ and $\bx'\in\cC^N$ with $\bx_{\omega^c} = \bx'_{\omega^c}$ and $\bx_{\omega} \neq \bx'_{\omega}$, as this implies that~\cref{eq:BRSC_for_KL} holds true with constant $C_K = \min_{\#\omega\leq K} C_\omega$. Now, let us fix $\omega$ such that $\#\omega \leq K$. To simplify~\cref{eq:BRSC_for_KL-2} we first lower-bound (resp., upper-bound) the left-hand-side (resp., the right-hand side) with simpler functions, exploiting i) the constraints on $\bx$ and $\bx'$ given $\omega$, ii) the non-negativity of $\bA$, $\bx$, $\bx'$, $\by$ and $\mathbf{b}$, and iii) the boundedness of $\cX$. Defining the following constants
$$
 \delta_1 := \frac{\min_{j\in \sigma_\by} y_j}{\max_{\bx \in \cX} \|\tilde{\bA}\bx + \mathbf{b}_{\sigma_\by}\|_2} >0, \quad \text{ and } \quad \delta_2 := \max_{\bx \in \cX} \|\tilde{\bA}_{\omega^c} \bx_{\omega^c} + \mathbf{b}_{\sigma_\by} \|_2>0,
$$
we get that the left hand side of~\cref{eq:BRSC_for_KL-2} satisfies
\begin{align}
    \sum_{j =1}^{\#\sigma_\by} \frac{y_{\sigma_\by(j)}[\tilde{\bA}_\omega(\bx_\omega-\bx'_\omega)]_j^2}{([\tilde{\bA}\bx]_j+b_{\sigma_\by(j)})([\tilde{\bA}_\omega\bx'_\omega]_j +[\tilde{\bA}_{\omega^c}\bx_{\omega^c}]_j +b_{\sigma_\by(j)})} & \geq \delta_1 \frac{\|\tilde{\bA}_\omega(\bx_\omega - \bx'_\omega)\|_2^2}{\|\tilde{\bA}\|_2 \|\bx'_\omega\|_2 + \delta_2} \notag \\  
    &\geq\delta_1 (1-\delta_K^-) \frac{\|\bx_\omega - \bx'_\omega\|_2^2}{\|\tilde{\bA}\|_2 \|\bx'_\omega\|_2 + \delta_2}.\label{eq:BRSC_for_KL-3}
\end{align}
for all $\bx\in\cX$ and all $\bx'\in\cC^N$ with $\bx_{\omega^c} = \bx'_{\omega^c}$ and $\bx_{\omega} \neq \bx'_{\omega}$, and where the last inequality follows by \cref{eq:LRIP}. 
Similarly, defining
$$
 \delta_3 := \left(\min_{\bx \in \cX} \min_{i \in \omega} x_i + \eta_i \right)^{-1} >0, 
$$
the right hand side of~\cref{eq:BRSC_for_KL-2} satisfies
\begin{equation}\label{eq:BRSC_for_KL-4}
     \sum_{i \in \omega} \frac{(x_i-x_i')^2}{(x_i+\eta_i)(x_i'+\eta_i)} \leq \delta_3 \sum_{i\in \omega} \frac{(x_i-x_i')^2}{x_i' + \eta_i} 
\end{equation}
for all $\bx$, $\bx'$ as above. Combining~\cref{eq:BRSC_for_KL-3} and~\cref{eq:BRSC_for_KL-4}, together with the existence of a sufficiently large $Q>0$ such that $\cX \subseteq [0,Q]^N$ ($\cX$ bounded), we get that a valid constant $C_\omega$ to have~\cref{eq:BRSC_for_KL-2} is given by
$$
     C_\omega := \frac{\delta_1 (1-\delta_K^-)}{\delta_3} \inf_{\substack{\bu \in [0,Q]^{k}, \bu' \in \R_{\geq 0 }^{k} \\ \bu\neq \bu', \boldsymbol{\Delta} = \bu - \bu'}}  \Gamma(\bu,\bu') :=\frac{\|\boldsymbol{\Delta}\|_2^2}{\|\tilde{\bA}\|_2 \|\bu'\|_2 + \delta_2}  \left( \sum_{i=1}^{k} \frac{\Delta_i^2}{u_i'+\eta_i}  \right)^{-1}
$$
where $k=\#\omega$ which, by the way, is the sole remaining dependence in $\omega$. 
As such, by~\cref{prop:orderBRSCconst},  we directly get a valid $C_K  = \min_{\#\omega\leq K} C_\omega$ by analyzing the above infimum for $k = K$.
Our goal now is to show that this constant is positive; that is, the ratio $\Gamma$ never vanishes on the considered constraint set. To that end, let $\bu\in[0, Q]^K$ and $\bu'\in\R_{\geq 0}^K$ with $\bu'\neq \bu$. We distinguish two cases: either $\|\bu'\|_2 \leq B$ or $\|\bu'\|_2> B$ with $B := 4\sqrt{K}Q$. 

\textit{Case $\|\bu'\|_2 \leq B$.} Define $\delta_4:=\min_{i=1,\ldots, K}\eta_i$ and observe that, as $\bu \neq \bu'$, then clearly $\Delta_i>0$ for at least one $i=1,\ldots, k$. We have that
$$
0 <\sum_{i=1}^K \frac{\Delta_i^2}{u_i' + \eta_i} \leq \frac{\|\boldsymbol{\Delta}\|_2^2}{\delta_4},
$$
leading to
$$
\Gamma(\bu,\bu') \geq \frac{\delta_4}{\|\tilde{\bA}\|_2 \|\bu'\|_2 + \delta_2} \geq \frac{\delta_4}{\|\tilde{\bA}\|_2 B  + \delta_2} >0,
$$
concluding the proof of the result. 

\textit{Case $\|\bu'\|_2 > B$.} By equivalence of norms in finite dimension, we have $\|\bu'\|_\infty \geq {\|\bu'\|_2}/{\sqrt{K}} > 4Q$. This implies that the set of indices $T := \{i : u'_i > 2Q\}$ contains at least one element. We denote $T^c$ its complement. Then observe that, for all $i \in T$,
\begin{equation}\label{eq:BRSC_for_KL-5}
   \Delta_i^2 = (u_i' - u_i)^2 \geq (u_i' - Q)^2 \geq \frac14 u_i'^2,
\end{equation}
where we used that $u_i \leq Q$ for all $i=1,\ldots, K$ and $ u_i' > 2Q$ for $i \in T$. It then follows that
$$
\|\boldsymbol{\Delta}\|_2^2 = \sum_{i \in T} \Delta_i^2 + \sum_{i \in T^c} \Delta_i^2 \geq \frac14  \sum_{i \in T} u_i'^2 = \frac14\left(\|\bu'\|_2^2 - \sum_{i \in T^c} u_i'^2\right).
$$
Then, by definition of $T^c$ we have that, for $i \in T^c$,  $u_i' \leq 2Q$ if and only if $- u_i'^2 \geq -4Q^2$, from which we deduce the bound
$$
\|\boldsymbol{\Delta}\|_2^2  \geq \frac14 \left(\|\bu'\|_2^2 - 4(\#T^c)Q^2 \right) \geq \frac14\left(\|\bu'\|_2^2 - 4KQ^2\right) \geq \frac{3}{16} \|\bu'\|_2^2,
$$
where we used that $(\#T^c) \leq K$ and $-4KQ^2 = - B^2/4  > -\|\bu'\|_2^2/4$. We now analyze the quantity $D(\bu,\bu') := \sum_{i=1}^{K} {\Delta_i^2}/(u_i'+\eta_i)$ involved in the definition of the ratio~$\Gamma$. Again, we split the indices of the sum in $T$ and $T^c$ to get
\begin{itemize}
    \item for all $i \in T$, we have that $\Delta_i^2 = (u_i' - u_i)^2 \leq u_i'^2$ (as $u_i' > 2Q > Q \geq u_i$) and 
    $$
    \sum_{i\in T} \frac{\Delta_i^2}{u_i'+\eta_i} \leq   \sum_{i\in T} \frac{u_i'^2}{u_i' + \eta_i} \leq \sum_{i\in T} u_i' \leq \|\bu'\|_1 \leq \sqrt{K}\|\bu'\|_2.
    $$
    \item for all $i \in T^c$, we have $u_i'+\eta_i \geq u_i'+\delta_4\geq\delta_4$  and $|\Delta_i| \leq u_i' + u_i \leq 3Q$, and
    $$
    \sum_{i\in T^c} \frac{\Delta_i^2}{u_i'+\eta_i} \leq \frac{(\#T^c)9Q^2}{\delta_4} \leq \delta_5 := \frac{9KQ^2}{\delta_4}. 
    $$
\end{itemize}
It then follows that
$$
D(\bu, \bu') \leq \sqrt{K}\|\bu'\|_2 + \delta_5.
$$
Combining all these bounds we obtain that, for all $\bu \in [0,Q]^K$ and $\|\bu'\|_2 > B$,
\begin{align*}
   \Gamma(\bu,\bu') &\geq \frac{3}{16} \frac{\|\bu'\|_2^2}{(\|\tilde{\bA}\|_2 \|\bu'\|_2 + \delta_2)(\sqrt{K}\|\bu'\|_2 + \delta_5)} \\
   &\geq \frac{3}{16} \frac{B^2}{(\|\tilde{\bA}\|_2 B + \delta_2)(\sqrt{K}B + \delta_5)} >0
\end{align*}
where we used the fact that the function $x \mapsto x^2/((ax +b)(cx+d))$ for $a,b,c,d>0$ is increasing on $\R_{\geq 0}$.

In all, we have shown that 
$$
C_K \geq \frac{\delta_1 (1-\delta^-_K)}{\delta_3} \min\left\{ \frac{\delta_4}{\|\tilde{\bA}\|_2 B  + \delta_2}, \frac{3}{16} \frac{B^2}{(\|\tilde{\bA}\|_2 B + \delta_2)(\sqrt{K}B + \delta_5)}\right\} >0
$$
which completes the proof.

\subsection{\texorpdfstring{Proofs of~\Cref{sec:sec4}}{Proofs of Section 4}}\label{sec:proofs32}

The proof of the  results stated in~\Cref{sec:uniquesparse} are based on the combination of~\cref{mainprop} and both \cref{lem:sec2lem1} and~\cref{lem:sec2lem2} presented below. These results generalize those of~\cite[Section 4.2]{Carlsson2020}, in the sense that the squared distance in the right hand side therein is substituted by the symmetric Bregman divergence with respect to $\Psi$.  To do so, we set ourselves within the assumptions of~\cref{thm:sparsest}: assume that $F=G_{\by}(\bA\cdot)$ satisfies \cref{eq:brsc} with respect to $\Psi$, $\cX$, and $K$. The proposition writes as follows.
\begin{proposition}\label{mainprop} For any two points $\bx\in\cX$, $\bx'\in \cC^N$, with $\bx\neq\bx'$, and such that $\|\bx-\bx'\|_0\leq K$, we have that 
\begin{equation}\label{eq:bregsparsity_bound}
       \langle \bz-\bz',\bx-\bx' \rangle\leq\left(1-C_K\right)D_\Psi^{\ \mathrm{symm}}(\bx, \bx'),
\end{equation}
where $\bz$, $\bz'$ are associated to $\bx$ and $\bx'$ as in \cref{eq:z}. 
\end{proposition}
\begin{proof}
As $F$ satisfies \cref{eq:brsc} for $K$, $\Psi$, and $\cX$, there exists a constant $C_K>0$ such that, for any $\bx\in\cX$, $\bx'\in \cC^N$,
$$
\begin{aligned}
\langle\bz-\bz', \bx-\bx'\rangle&= \langle \nabla\Psi(\bx)-\nabla\Psi(\bx'), \bx-\bx'\rangle\\
&\quad-\langle \bA^T\nabla G_{\by}(\bA\bx)-\bA^T\nabla G_{\by}(\bA\bx'), \bx-\bx'\rangle\\
&=D_\Psi^{\ \mathrm{symm}}(\bx, \bx')-D_F^{\ \mathrm{symm}}(\bx,\bx')\\
&\leq\left(1-C_K\right)D_\Psi^{\ \mathrm{symm}}(\bx, \bx'),
\end{aligned}
$$
which is the inequality we were aiming for.
\end{proof}

Note that~\cref{thm:sparsest} and~\cref{thm:uniqueglob} hold for any $C_K>0$. While the case $C_K\geq 1$ admits a straightforward proof, the setting $C_K\in(0, 1)$ is more delicate and requires the following two lemmas, which are thus dedicated exclusively to this regime.
\begin{lemma}\label{lem:sec2lem1}
Let $\bx\in\cC^N$ and consider $\bz$ as in \cref{eq:z}. Assume that, for some $i\in \{1,\ldots, N\}$, we have $z_i\in\partial h_i(x_i)$ and
\begin{equation}\label{eq:condlem1}
|z_i|> \frac{\psi_i'(\alpha_i)- \psi_i'(0)}{C_K} + \psi_i'(0).
\end{equation}
Then, for any $\bx'\in\cC^N$ with $x_i' \neq x_i$ and $z_i'\in\partial h_i(x'_i)$, we have that
$$
(z_i-z_i')(x_i-x_i')> \left(1-C_K\right)d^{\ \mathrm{symm}}_{\psi_i}(x_i, x_i').
$$
\end{lemma}
Notice that the notation $d^{\ \mathrm{symm}}_{\psi}$ simply refers to the one-dimensional symmetric Bregman divergences with respect to $\psi$.
\begin{proof} Let $i\in \{1,\ldots, N\}$ be such that \cref{eq:condlem1} holds. Looking at the definition of $\partial h_i$ in \cref{eq:subdiffg_i}, we see that it is an odd map, and therefore it is enough to consider the case in which $x_i, z_i\geq 0$. Combining this with \cref{eq:condlem1},  $C_K<1$, and $\psi_i'(\alpha_i) > \psi'_i(0)$ ($\psi_i'$ increasing), we get that
$$
z_i> \frac{\psi_i'(\alpha_i)- \psi_i'(0)}{C_K} + \psi_i'(0)> \psi_i'(\alpha_i)
$$ 
Then, necessarily $z_i=\psi'_i(x_i)$ and so $x_i>\alpha_i$. We next aim to minimize the quotient
$$
Q:=\frac{(z_i-z_i')(x_i-x_i')}{d_{\psi_i}^{\ \mathrm{symm}}(x_i, x_i')}= \frac{z_i - z_i'}{\psi_i'(x_i) - \psi_i'(x_i')},
$$
and show that it is greater than $1-C_K$. To do so, and motivated by the expression of the subgradient given in \cref{eq:subdiffg_i}, we distinguish three cases. If $x_i'=0$, then by \cref{eq:subdiffg_i} we know that $z_i'\in[-\psi'_i(\alpha_i), \psi'_i(\alpha_i)]$. Then,
$$
Q=\frac{z_i-z_i'}{\psi_i'(x_i)-\psi'_i(0)}\geq\frac{z_i-\psi_i'(\alpha_i)}{z_i-\psi'_i(0)}\geq 1+\frac{\psi'_i(0)-\psi'_i(\alpha_i)}{z_i-\psi_i'(0)},
$$
and the term in the right hand side is greater than $1-C_K$ if and only if \cref{eq:condlem1} is satisfied. Next, if $0<x_i'\leq\alpha_i$, we get that $z_i'=\psi_i'(\alpha_i)$ and so 
$$
Q=\frac{z_i - \psi_i'(\alpha_i)}{z_i - \psi_i'(x_i')}.
$$
Now, we fix $x'_i$ and observe that, as a function of $z_i$, $Q$ is minimized for the smallest admissible $z_i$, i.e., $z_i=(\psi_i'(\alpha_i)- \psi_i'(0))/C_K + \psi_i'(0)$. Then,
$$
Q>\frac{(\psi_i'(\alpha_i)-\psi'_i(0))/C_K+\psi_i'(0)-\psi'_i(\alpha_i)}{(\psi_i'(\alpha_i)-\psi'_i(0))/C_K+\psi_i'(0) - \psi_i'(x_i')},
$$
where both the numerator and the denominator are positive because, as $C_K<1$, we have that $(\psi_i'(\alpha_i)-\psi'_i(0))/C_K\geq \psi_i'(\alpha_i)-\psi'_i(0)>0$ and,  as $\psi'_i$ is increasing and $0<x'_i\leq\alpha_i$, that $(\psi_i'(\alpha_i)-\psi'_i(0))/C_K \geq \psi_i'(x_i')-\psi'_i(0)>0$. This also implies that $(\psi_i'(\alpha_i)-\psi'_i(0))/C_K > (\psi_i'(\alpha_i)-\psi'_i(0))/C_K +\psi'_i(0)- \psi_i'(x_i')>0$. If we plug the latter into the above lower bound for $Q$, we conclude that
$$
Q> \frac{(\psi_i'(\alpha_i)-\psi'_i(0))/C_K+\psi_i'(0)-\psi'_i(\alpha_i)}{(\psi_i'(\alpha_i)-\psi'_i(0))/C_K}= 1- C_K.
$$
Finally, if $x_i'>\alpha_i$, then $z_i'=\psi_i'(x_i')$ leading to $z_i'-z_i=\psi_i'(x_i')-\psi_i'(x_i)$ and $Q=1$. As $C_K<1$, we conclude.
\end{proof}

\begin{lemma}\label{lem:sec2lem2}
Let $\bx\in \cC^N$ and consider $\bz$ as in \cref{eq:z}. Assume that, for some $i \in \{1,\ldots, N\}$, we have $z_i\in\partial h_i(x_i)$ and
\begin{equation}\label{eq:condlem2}
|z_i|< C_K(\psi_i'(\alpha_i)-\psi_i'(0)) + \psi_i'(0).
\end{equation}
Then, for any $\bx'\in\cC^N$ with $x_i' \neq x_i$ and $z_i'\in\partial h_i(x_i')$, we have that
$$
(z_i-z_i')(x_i-x_i')> \left(1-C_K\right)d^{\ \mathrm{symm}}_{\psi_i}(x_i, x_i').
$$
\end{lemma}
\begin{proof} Let $i\in \{1,\ldots, N\}$ be such that \cref{eq:condlem2} holds. As in the proof of~\cref{lem:sec2lem1}, we aim to minimize the quantity
$$
Q:=\frac{z_i'-z_i}{\psi_i'(x_i')-\psi_i'(x_i)}
$$
and show that it is greater than $1-C_K$. For the same reasons as in the proof of~\cref{lem:sec2lem1}, we assume that $x_i,z_i\geq0$. Combining both the fact that $C_K<1$, $\psi_i'(\alpha_i) > \psi'_i(0)$ ($\psi_i'$ increasing), and \cref{eq:condlem2}, we derive that
$$
z_i<C_K(\psi_i'(\alpha_i)-\psi_i'(0)) + \psi_i'(0)<\psi'_i(\alpha_i),
$$
and so, necessarily $x_i=0$. Since by assumption we have $x_i' \neq x_i$, necessarily $x_i' \neq  0$. We then distinguish two cases for $x_i'$: either $0<x_i'\leq \alpha_i$ or $x_i'> \alpha_i$. In the first case, $z_i'=\psi_i'(\alpha_i)$ and, since $\psi_i'$ is increasing, we have that
$$
Q\geq\frac{\psi_i'(\alpha_i)-C_K(\psi_i'(\alpha_i)-\psi_i'(0))-\psi_i'(0)}{\psi_i'(\alpha_i)-\psi_i'(0)}=1-C_K
$$
as desired. 

In the second case ($x_i'> \alpha_i$) we have that $z_i'=\psi_i'(x'_i) \geq \psi_i'(\alpha_i)$ (again $\psi'_i$ increasing) which leads to
$$
Q>\frac{\psi_i'(x'_i)-C_K(\psi_i'(x'_i)-\psi_i'(0))-\psi_i'(0)}{\psi_i'(x_i')-\psi_i'(0)}=1-C_K.
$$
This concludes the proof.
\end{proof}
We are now ready to prove~\cref{thm:sparsest}. 

\subsubsection{\texorpdfstring{Proof of~\cref{thm:sparsest}}{Proof of Theorem 3.1}}\label{proof:thm41}
Let $\bx\in \mathcal{X}$ be a critical point of $J_\Psi$. The proof of this result will be based on applying~\cref{mainprop} backwards; i.e., we will show that for any other critical point $\bx'\in\cC^N$, $\bx\neq\bx'$, of $J_\Psi$, we have
$$
\langle \bz-\bz',\bx-\bx' \rangle > \left(1-C_K\right)D_\Psi^{\ \mathrm{symm}}(\bx, \bx').
$$
Then, by~\cref{mainprop}, this necessarily implies that $\|\bx-\bx'\|_0>K$. We distinguish three cases. $1-C_K<0$, $1-C_K=0$, and $1-C_K>0$.  If $1-C_K<0$, as the right hand side above turns out to be negative, it is enough to show that $\langle \bz-\bz',\bx-\bx' \rangle\geq 0$. As the inner product is separable, it is sufficient to show that $(z_i-z_i')(x_i-x_i')\geq 0$ for all $i=1,\ldots, N$, as summing over $i$ gives the result. As $z_i \in\partial h_i(x_i)$, $z_i' \in\partial h_i(x_i')$ and each $\partial h_i$ is a monotone operator (from convexity of $h_i$), the desired bound holds true by definition. 

Next, if $1-C_K=0$, it suffices to show that $\langle \bz-\bz',\bx-\bx' \rangle > 0$. To do so, and following an element-wise analysis as above, we first point out that, by definition, the $\partial h_i$ are odd maps. Hence, we may assume, without loss of generality, that $x_i, \ x_i'\geq 0$ and $z_i, \ z_i'\geq 0$. As $C_K=1$, \cref{eq:zinotin} implies that $z_i\neq \psi_i'(\alpha_i)$. Then, either $x_i=0$ or $x_i>\alpha_i$. 
\begin{itemize}
    \item If $x_i=0$, as $z_i\in\partial h_i(x_i)$, \cref{eq:subdiffg_i} gives $z_i\in[0, \psi_i'(\alpha_i))$. As $x_i'\neq x_i$ by assumption, we get that $x_i'\neq 0$, which again by \cref{eq:subdiffg_i} gives $z'_i\geq \psi'_i(\alpha_i)$. Hence, $z_i-z_i'<0$ and we get that $(z_i-z_i')(x_i-x_i')=(z_i-z_i')(-x_i')>0$.
    \item If $x_i>\alpha_i$, \cref{eq:subdiffg_i} implies that $z_i=\psi_i'(x_i)$. Now, for each $i=1,\ldots, N$ with $x'_i\neq x_i$, we have either $x_i'\leq \alpha_i$ or $x_i'>\alpha_i$. If $x_i'\leq\alpha_i$, we get directly that $x'_i<x_i$ and \cref{eq:subdiffg_i} gives $z_i'\leq\psi'_i(\alpha_i) < \psi'_i(x_i) = z_i$ (again using $\psi'_i$ increasing). Hence, $(z_i-z_i')(x_i-x_i')> 0$. Finally, if $x_i'>\alpha_i$,  we have that $z_i' = \psi_i'(x_i')$ and either $x_i > x_i'$ or viceversa. Then, from the strict convexity of $\psi_i$, $(z_i-z_i')(x_i-x_i') = (\psi_i'(x_i) - \psi_i'(x_i'))(x_i-x_i')> 0$. Since the case $x_i'=x_i$ gives $(z_i-z_i')(x_i-x_i')=0$, this term does not contribute to the sum over $i$.
\end{itemize}
This concludes this part of the proof.

Finally, if $1-C_K>0$ we have, by both~\cref{lem:sec2lem1} and~\cref{lem:sec2lem2}, that $(z_i-z_i')(x_i-x_i')>(1-C_K)d^{\ \mathrm{symm}}_{\psi_i}(x_i, x_i')\geq 0$ for all $i$ such that $x'_i\neq x_i$, since the case $x_i'=x_i$ gives $(z_i-z_i')(x_i-x_i')=0$ as above. With this, we conclude the proof of the theorem. 

\subsubsection{\texorpdfstring{Proof of~\cref{thm:uniqueglob}}{Proof of Theorem 3.3}}\label{proof_thm:uniqueglob}

Let $\bx$ be a critical point of $J_\Psi$ satisfying the conditions of the theorem and set $k = \|\bx\|_0$. Assume that $\bx$ is not the unique global minimizer of $J_\Psi$. Hence, there exists a local minimizer $\bx'\in \cC^N$, $\bx' \neq \bx$, such that $J_\Psi(\bx')\leq J_\Psi(\bx)$. Then by \cite[Proposition 10]{Essafri2024}, we have that $J_\Psi(\bx')=J_0(\bx')$. Moreover, for the critical point $\bx$ we have $J_\Psi(\bx) \leq J_0(\bx)$ as $B_\Psi \leq \lambda_0 \|\cdot\|_0$ by definition. Finally, by~\cref{thm:sparsest}, we have that
$$
K<\|\bx'-\bx\|_0\leq \|\bx'\|_0+\|\bx\|_0=\|\bx'\|_0+k,
$$
and so necessarily $\|\bx'\|_0\geq K-k+1$. 

Combining all these inequalities, it follows
\begin{align*}
    J_\Psi(\bx')-J_\Psi(\bx)& \geq  J_0(\bx')-J_0(\bx) \\
    & \geq G_\by(\bA\bx') + \lambda_0(K-k+1) - G_\by(\bA\bx)  - \lambda_0 k \\
    & \geq \lambda_0(K-2k+1) - G_\by(\bA\bx) \underset{\cref{eq:uniqglobcond}}{>} 0.
\end{align*}
This implies that $J_\Psi(\bx')>J_\Psi(\bx)$, which contradicts or initial assumption on $\bx'$. Finally, by~\cref{thm:exactrel}, $\bx$ is also the unique global minimizer of $J_0$. 

\subsection{\texorpdfstring{Proofs of~\Cref{sec:oracsol}}{Proofs of section 4}}\label{sec:proofs33}

We provide in this section the proof of~\cref{thm:x'locmin} and~\cref{thm:oracuniquesparsest}. In contrast with the previous section (i.e.~\Cref{sec:proofs32}), extending the results of~\cite[Section 4.3]{Carlsson2020} is much more involved, since their analysis crucially relies on the fact that they have access to a closed-form expression of the oracle solution. As this is not available in our setting, the methods therein must be suitably adapted. We start with the following general lemma, which will be useful in the proof of both results, and is a direct consequence of the so-called descent lemma.

\begin{lemma}\label{lem:coseqdesclemm} Under~\Cref{ass:FID},  there exists a constant $\tilde{L} \geq L$ such that, for any $\mathbf{w} \in \mathrm{Im}(\bA |_{\cC^N})$,
\begin{equation}\label{eq:coseqdesclemm}
    \|\nabla G_\by(\mathbf{w})\|\leq \sqrt{2 \tilde{L} G_\by(\mathbf{w})}.
\end{equation}
In particular, when $\dom G_\by = \R^M$, $\tilde{L} = L$ is valid.
\end{lemma}
\begin{proof}
Since $G_\by$ has $L$-Lispchitz gradient by assumption, it satisfies the descent lemma \cite[Theorem 18.15 (iii)]{BCombettes}, which writes, for any $\mathbf{w},\bz \in \mathrm{int}(\dom G_\by)$, 
$$
G_\by(\bz)\leq G_\by(\mathbf{w})+\nabla G_\by(\mathbf{w})^T(\bz-\mathbf{w})+\frac{L}{2}\|\bz-\mathbf{w}\|_2^2.
$$
Now, we claim that there exists $\eta >0$ such that, for all $\mathbf{w}\in\mathrm{int}(\dom G_\by)$, $\bz:= \mathbf{w}-\eta\nabla G_\by(\mathbf{w}) \in \mathrm{int}(\dom G_\by)$. Then we derive that
\begin{align*}
    0\leq G_\by(\mathbf{w}-\eta\nabla G_\by(\mathbf{w}))& \leq G_\by(\mathbf{w})-\eta\|\nabla G_\by(\mathbf{w})\|_2^2+\frac{L\eta^2}{2}\|\nabla G_\by(\mathbf{w})\|_2^2, \\
    & = G_\by(\mathbf{w}) -  \left(\eta-\frac{L\eta^2}{2} \right)\|\nabla G_\by(\mathbf{w})\|_2^2
\end{align*}
where we note that the left hand side above is non-negative by assumption. Taking $\eta < \tfrac2L$ and rearranging the terms, we conclude that
\begin{equation}\label{eq:proof_coseqdesclemm-1}
    \|\nabla G_\by(\mathbf{w})\|_2^2\leq  2\left(2\eta-L\eta^2 \right)^{-1} G_\by(\mathbf{w}).
\end{equation}
In particular, one can observe that the concave quadratic function $\eta \mapsto 2\eta-L\eta^2$ attains its maximum at $\eta = \tfrac1L$, which would be the choice leading to the best (smallest) bound. 

To complete the proof, it remains to prove the claim; i.e., that there exists $\eta \in (0, \tfrac1L]$ such that $\bz = \mathbf{w}-\eta\nabla G_\by(\mathbf{w}) \in \mathrm{int}(\dom G_\by)$  for all $\mathbf{w}\in\mathrm{int}(\dom G_\by)$. While this is trivial when $\dom G_\by = \R^M$ (all $\eta$ are fine), more care is needed when $\dom G_\by \subset \R^M$.  Let $\mathbf{w}\in\mathrm{int}(\dom G_\by)$. First, from \Cref{ass:fid3} of \Cref{ass:FID}, we can define a uniform margin,
$$
\delta := \inf_{\mathbf{w} \in \mathrm{Im}(\bA |_{\cC^N})} \mathrm{dist}(\mathbf{w}, \partial (\dom G_\by) ) >0
$$
where $\partial (\dom G_\by) $ denotes the boundary of the domain of $G_\by$. Additionally, from \Cref{ass:fid4} of \Cref{ass:FID} we have $\theta := \sup_{\mathbf{w}\in \mathrm{Im}(\bA |_{\cC^N}) } \|\nabla G_\by(\mathbf{w}) \| < \infty$. Taking $\eta < \tfrac{\delta}{\theta}$ leads to
$$
\|\bz - \mathbf{w} \|=\|\eta\nabla G_\by(\mathbf{w})\| < \frac{\delta}{\theta} \|\nabla G_\by(\mathbf{w})\| \leq \delta,
$$
which in turn implies that  $\bz \in\mathrm{int}( \dom G_\by)$ since, if $\bz \notin\mathrm{int}( \dom G_\by)$, then necessarily $\|\bz - \mathbf{w} \|\geq \delta$. Hence, we have shown that there exists $\eta < \min\left\{\frac1L, \frac{\delta}{\theta} \right\}$ such that~\eqref{eq:proof_coseqdesclemm-1} holds true, which can be equivalently expressed as the existence of $\tilde{L} \geq L$ such that~\eqref{eq:coseqdesclemm} holds. The fact that $\tilde{L} = L$ is valid when $\dom G_\by = \R^M$ is direct from the above derivations.
\end{proof}


We are now ready to prove~\cref{thm:x'locmin}.

\subsubsection{\texorpdfstring{Proof of~\cref{thm:x'locmin}}{Proof of Theorem 4.4}}\label{sec:proofxorlocmin}

As in the proof sketch provided in the main text, we recall, from \cite[Proposition 10]{Essafri2024}, that a local minimizer $\bx$ of $J_0$ is preserved by $J_\Psi$ if and only if
\begin{enumerate}[label=(\roman*),leftmargin=*]
\item \label{cond:suppcomps-proof} for any $i\in \sigma(\bx)$, $|x_i|>\alpha_i$,
\item \label{cond:offsuppcomps-proof} for any $i\in{\sigma}(\bx)^c$, $-\langle \ba_i, \nabla G_\by(\bA\bx)\rangle\in [\ell_i^-, \ell_i^+]$,
\end{enumerate}
where, as already mentioned, $\ell_i^+:=\psi'_i(\alpha_i)-\psi'_i(0)$ and either $\ell_i^-:=-\psi'_i(\alpha_i)-\psi'_i(0)$ if $\cC=\R$ or $\ell_i^-=-\infty$ if $\cC=\R_{\geq 0}$. Then, the proof of the result will be divided into two  parts. First, we show that~\cref{eq:onsup} is sufficient for~\cref{cond:suppcomps-proof} to hold true at $\bx^{\mathrm{or}}$. Second, we will show that, if~\cref{eq:offsup} holds, then the off-support~\cref{cond:offsuppcomps-proof} is satisfied at $\bx^{\mathrm{or}}$.

Let $S$ be a safe oracle region and assume that $S\cap\Omega= \emptyset$. We start by showing that $|x^\mathrm{or}_i|>\alpha_i$ for each $i\in\sigma(\bx^\mathrm{or})$. We recall that, by definition, $\bx^\mathrm{or}=Z_{\sigma^*}(\bu^\mathrm{or})$, with $\bu^\mathrm{or}$ given by \cref{eq:oracsol}. We aim to show that $|x^\mathrm{or}_{\sigma^*[j]}|=|u^\mathrm{or}_j|>\alpha_{\sigma^*[j]}$ for each $j=1,\ldots, k^*$, and we proceed by contradiction; i.e., we assume that there exists $j_0\in\{1,\ldots, k^*\}$ such that $|u^\mathrm{or}_{j_0}|\leq\alpha_{\sigma^*[j_0]}$. Then,  $ \bu^\mathrm{or}\in\Omega_{\sigma^*[j_0]}\subset \Omega$. As we know by construction that $ \bu^\mathrm{or}\in S$, we have shown that $ \bu^\mathrm{or}\in S\cap\Omega$, a contradiction with the assumption $S\cap\Omega = \emptyset$. Finally, the fact that $\sigma(\bx^\mathrm{or})=\sigma^*$ is a direct consequence of what we have just proved, as we eliminate the possibility of having a component $i\in\sigma^*$ with $x^\mathrm{or}_i=0$. This concludes the first part of the proof.

We now show that \cref{eq:offsup} is a sufficient condition for  \Cref{cond:offsuppcomps-proof} above to be satisfied at $\bx^{\mathrm{or}}$. Notice that, by the Cauchy--Schwartz inequality, we have, for all $i\in{\sigma^*}^c$, that
$$
\begin{aligned}
    |\langle \ba_i, \nabla G_\by(\bA\bx^{\mathrm{or}})\rangle|&\leq \|\ba_i\|_2\|\nabla G_\by(\bA\bx^{\mathrm{or}})\|_2\\
    &\leq \|\ba_i\|_2\sqrt{2\tilde L G_\by(\bA\bx^{\mathrm{or}})}\\
    &= \|\ba_i\|_2\sqrt{2\tilde L F(\bx^{\mathrm{or}})}\\
    &\leq \|\ba_i\|_2\sqrt{2\tilde L F(\bx^*)},
\end{aligned}
$$
where we applied~\cref{lem:coseqdesclemm} and where the last inequality follows by the definition of $\bx^{\mathrm{or}}$. Now, we know by~\cref{eq:offsup} that, for all $i\in{\sigma^*}^c$,
\begin{equation}\label{eq:thm39comb2}
\|\ba_i\|_2\sqrt{2\tilde L F(\bx^*)}\leq \psi'_i(\alpha_i)-\psi'_i(0) \leq \psi'_i(\alpha_i)+\psi'_i(0).
\end{equation}
Using both of the estimates derived in \cref{eq:thm39comb2}, we conclude that
$$
- \langle \ba_i, \nabla G_\by(\bA\bx^{\mathrm{or}})\rangle\leq \|\ba_i\|_2\sqrt{2\tilde L F(\bx^*}) \leq \psi'_i(\alpha_i)-\psi'_i(0)=\ell_i^+,
$$
and
$$
- \langle \ba_i, \nabla G_\by(\bA\bx^{\mathrm{or}})\rangle\geq -\|\ba_i\|_2\sqrt{2\tilde L F(\bx^*)} \geq -\psi'_i(\alpha_i)-\psi'_i(0),
$$
for all $i\in{\sigma^*}^c$, where the term in the right hand side $-\psi'_i(\alpha_i)-\psi'_i(0)$ is precisely $\ell_i^-$ when $\cC=\R$ and, when $\cC=\R_{\geq 0}$, we have $-\psi'_i(\alpha_i)-\psi'_i(0)>\ell_i^-=-\infty$, concluding. 

\subsubsection{\texorpdfstring{Proof of~\cref{thm:oracuniquesparsest}}{Proof of Theorem 4.6}}\label{proof:oracuniquesparsest}

As already mentioned in the sketch, the proof of this result will consist in combining~\cref{thm:uniqueglob} and~\cref{thm:x'locmin}. To do so, we will distinguish two cases: when $C_K>1$ and when $C_K\leq 1$. 

\emph{The case $C_K>1$.} First, the required BRSC condition of~\cref{thm:x'locmin} at $\bx^{\mathrm{or}}$ holds by~\cref{ass:BRSCor}. Now, observe that we can directly apply~\cref{thm:x'locmin}, as in this case we ask for conditions~\cref{eq:onsup} and~\cref{eq:offsup} to hold true. Then, $\bx^{\mathrm{or}}$ is a local minimizer (and so a critical point) of $J_\Psi$ with $\sigma(\bx^\mathrm{or})=\sigma^*$. Next, as $C_K>1$, Condition~\cref{eq:zinotin} reduces to $|z_i| \notin \emptyset$ and is thus trivially satisfied. Finally, setting
\begin{equation}\label{eq:lamb0lowboundmaintheorem}
    \lambda_0 > \frac{F(\bx^*)}{1+K-2k^*} \;\underset{\text{Def of }\bx^{\mathrm{or}}}{\geq } \; \frac{F(\bx^{\mathrm{or}})}{1+K-2k^*} 
\end{equation}
completes the requirements of~\cref{thm:uniqueglob}, and concluding this part of the result.

\emph{The case $C_K\leq 1$.} We start by noting that, as in the above case, the required BRSC condition of~\cref{thm:x'locmin} at $\bx^{\mathrm{or}}$ holds by~\cref{ass:BRSCor}. Next, we will prove that conditions~\cref{ass1genthm} and~\cref{ass2genthm} imply conditions~\cref{eq:onsup} and~\cref{eq:offsup} of~\cref{thm:x'locmin}. We will first show that $S\cap\Omega=\emptyset$ and we proceed by contradiction. Assume that $S\cap\Omega \neq \emptyset$ and let $\bu\in S\cap\Omega$. In particular $\bu\in\Omega$ and thus there exists $j \in\{1,\ldots,k^*\}$ such that $|u_j|\leq \alpha_{\sigma^*[j]}$. However, as $S\cap\tilde{\Omega}=\emptyset$ by assumption, we have
$$
|\psi_{\sigma^*[j]}'(u_j)|> \rho_j.
$$
 We now distinguish two cases: either $\psi_{\sigma^*[j]}'(u_j)> \rho_j$ or $\psi_{\sigma^*[j]}'(u_j)<- \rho_j$. On the one hand, if $\psi_{\sigma^*[j]}'(u_j)> \rho_j$, we have that, since $C_K\leq1$, 
$$
\psi_{\sigma^*[j]}'(u_j)> \frac{\psi_{\sigma^*[j]}'(\alpha_{\sigma^*[j]})-\psi'_{\sigma^*[j]}(0)}{C_K}+\psi_{\sigma^*[j]}'(0) \geq \psi_{\sigma^*[j]}'(\alpha_{\sigma^*[j]}).
$$

Using the fact that $\psi_{\sigma^*[j]}'$ is increasing we get that $u_j > \alpha_{\sigma^*[j]}$, leading to a contradiction. On the other hand, if $\psi_{\sigma^*[j]}'(u_j)<- \rho_j$, we have that 
$$
\psi_{\sigma^*[j]}'(u_j)<- \frac{\psi_{\sigma^*[j]}'(\alpha_{\sigma^*[j]})-\psi'_{\sigma^*[j]}(0)}{C_K}-\psi_{\sigma^*[j]}'(0) \leq -\psi_{\sigma^*[j]}'(\alpha_{\sigma^*[j]})
$$
using again that $C_K \leq 1$. As before, recalling that $-\psi_{\sigma^*[j]}'(\alpha_{\sigma^*[j]})=\psi_{\sigma^*[j]}'(-\alpha_{\sigma^*[j]})$ by~\cref{ass3psi} in~\cref{ass:generating_func_psi}, and using the fact that $\psi_{\sigma^*[j]}'$ is increasing, we get a contradiction. Hence, we have shown that~\cref{ass1genthm} (i.e., $S \cap \tilde{\Omega} = \emptyset$) implies~\cref{eq:onsup} (i.e., $S \cap \Omega = \emptyset$). Next, notice that~\cref{ass2genthm} directly implies~\cref{eq:offsup} because $C_K\leq 1$. We have thus all the conditions of \cref{thm:x'locmin} showing that $\bx^\mathrm{or}$ is a local minimizer of $J_\Psi$ and $\sigma(\bx^\mathrm{or}) = \sigma^*$.

Last, we will show that conditions~\cref{ass1genthm} and~\cref{ass2genthm} also imply condition~\cref{eq:zinotin} of~\cref{thm:sparsest}. Recalling that $z^\mathrm{or}_i=\psi_i'(x^\mathrm{or}_i)-\langle \ba_i, \nabla G_{\by}(\bA\bx^\mathrm{or})\rangle$, we see that~\cref{eq:zinotin} in~\cref{thm:sparsest} writes
\begin{equation}\label{eq:onsuppgenthm}
|\psi'_i(x^\mathrm{or}_i)|\notin \left[C_K(\psi_i'(\alpha_i)-\psi'_i(0))+\psi_i'(0), \frac{\psi_i'(\alpha_i)-\psi'_i(0)}{C_K}+\psi_i'(0)\right]
\end{equation}
for all $i\in\sigma^*$ (because, by~\cref{eq:zerogradientcond}, $\langle\ba_i, \nabla G_{\by}(\bA\bx^\mathrm{or})\rangle=0$ for all $i\in\sigma(\bx^{\mathrm{or}})=\sigma^*$)  and
\begin{equation}\label{eq:offsuppgenthm}
|\langle \ba_i, \nabla G_{\by}(\bA\bx^\mathrm{or})\rangle -\psi_i'(0)|\notin \left[C_K(\psi_i'(\alpha_i)-\psi'_i(0))+\psi_i'(0), \frac{\psi_i'(\alpha_i)-\psi'_i(0)}{C_K}+\psi_i'(0)\right]
\end{equation}
for all $i\in{\sigma^*}^c$. First, observe that~\cref{ass1genthm} in~\cref{thm:oracuniquesparsest} implies that
$$
|\psi_i'(x^\mathrm{or}_i)|>\frac{\psi_i'(\alpha_i)-\psi'_i(0)}{C_K}+\psi'_i(0)
$$
for each $i\in\sigma^*$, showing that \cref{eq:onsuppgenthm} holds. Now, following the same argument as in the proof above, we have that, by~\cref{lem:coseqdesclemm},
$$
|\langle \ba_i, \nabla G_\by(\bA\bx^{\mathrm{or}})\rangle|\leq \|\ba_i\|_2\sqrt{2\tilde L F(\bx^*)}, \text{ for all } i\in{\sigma^*}^c,
$$
and so, plugging condition~\cref{ass2genthm} to the estimate above we have that, for each $i\in{\sigma^*}^c$,
$$
|\langle \ba_i, \nabla G_\by(\bA\bx^{\mathrm{or}})\rangle|\leq C_K(\psi'_i(\alpha_i) - \psi'_i(0)).
$$
Then, for all $i\in{\sigma^*}^c$ it holds that
$$
\begin{aligned}
|\langle \ba_i, \nabla G_\by(\bA\bx^{\mathrm{or}})\rangle-\psi_i'(0)|&\leq |\langle \ba_i, \nabla G_\by(\bA\bx^{\mathrm{or}})\rangle|+\psi_i'(0)\\
&\leq C_K(\psi'_i(\alpha_i) - \psi'_i(0))+ \psi_i'(0)
\end{aligned}
$$
and \cref{eq:offsuppgenthm} follows. Finally, and as in the case $C_K>1$, condition~\cref{eq:lamb0lowboundmaintheorem} completes the requirements of~\cref{thm:uniqueglob} to be applied, and so, we conclude that $\bx^\mathrm{or}$ is the unique global minimizer of $J_\Psi$ and  any other critical point $\bx'$ of $J_\Psi$, $\bx'\neq\bx^\mathrm{or}$, is such that  $\|\bx'-\bx^\mathrm{or}\|_0>K$.   

\subsection{\texorpdfstring{Proofs of~\Cref{sec:applications}}{Proofs of section 5}}\label{sec:proofsappls}

\subsubsection{\texorpdfstring{Proof of~\cref{coro:general_result_l2}}{Proof of Corollary 5.5}}\label{proof:corolgenL2}

The proof consists in simplifying the conditions of~\cref{thm:oracuniquesparsest} to the present setting.   
First of all, notice that, with $\Psi$ given by~\cref{eq:Psi_l2}, we have that $\alpha_i = \alpha = \sqrt{2\lambda_0/\gamma}$ for all $i$ (see~\cite[Table 3]{Essafri2024}) and $\psi_i'(\alpha_i) = \gamma \alpha_i = \sqrt{2\lambda_0 \gamma}$. We now distinguish the two cases. If $C_K > 1$, then Condition~\cref{eq:onsup} is equivalent to
$$
\sqrt{\frac{2\lambda_0}{\gamma}} < | u_j^* | - \sqrt{\frac{2 F_{\sigma^*}(\bu^*)}{\gamma C_K} }, \; \text{for all } j = 1,\ldots,k^*,
$$
and~\cref{eq:offsup} writes
$$
\sqrt{2\lambda_0 \gamma} >  \| \ba_i \|_2\sqrt{2\tilde LF(\bx^*)},  \, \text{for all } i \in \sigma^{*c},
$$
On the other hand, if $C_K \leq 1$, we get that Condition~\cref{ass1genthm} is equivalent to
$$
\sqrt{\frac{2\lambda_0}{\gamma}} < C_K\left(| u_j^* | - \sqrt{\frac{2 F_{\sigma^*}(\bu^*)}{\gamma C_K} }\right), \; \text{for all } j = 1,\ldots,k^*,
$$
and, finally, that~\cref{ass2genthm} reads
$$ 
\sqrt{2\lambda_0 \gamma} >  \frac{\| \ba_i \|_2\sqrt{2\tilde LF(\bx^*)}}{C_K},  \, \text{for all } i \in \sigma^{*c}.
$$
Isolating $\lambda_0$, recalling that $F_{\sigma^*}(\bu^*) = F(\bx^*)$,
and combining all the above conditions together completes the proof.

\subsubsection{\texorpdfstring{Proof of~\cref{coro:KL_application}}{Proof of Corollary 5.9}}\label{proofcorolKL}

Recall that, with $\Psi$ given by~\cref{eq:PsiKL}, we have that
\begin{equation}\label{eq:alphaKL}
\alpha_i=\frac{-1}{c_i}\left(\frac{\xi}{W(-e^{-\kappa_i})}+\xi\right), \quad \text{for all }  i=1,\ldots, N,
\end{equation}
where $W$ denotes the Lambert function (see e.g.~\cite{Mezo2022}) and $\kappa_i=\lambda_0/(\gamma_i\xi)+1$. We start by computing  $\underline{\Lambda}$. To do so, we recall that the (principal branch of the) Lambert function $W$ satisfies both $W(-e^{-1})=-1$ and $W(0)=0$. Combining this with the fact that it is concave, we have that $W(x)\geq ex$ for all $x \in [-e^{-1},0]$. Then, as $\kappa_i>1$ for all $i=1,\ldots, N$, we have that
$$
 W(-e^{-\kappa_i}) \geq -e^{1-\kappa_i},
$$
which is equivalent to
$$
\frac{\xi}{W(-e^{-\kappa_i})} + \xi \leq -\xi e^{\kappa_i-1} + \xi, \quad \text{for all } i=1,\ldots, N.
$$
Recalling the value of $\alpha_i$ given above, we have that, for all $i=1,\ldots, N$,
$$
\alpha_i \geq \frac{\xi}{c_i}( e^{\kappa_i-1}  - 1) = \frac{\xi}{c_i}( e^{\lambda_0/(\gamma_i\xi)} - 1).
$$
Now, as $\psi_i'(x) = \gamma_i c_i \left(1 - \xi/(c_ix +\xi) \right)$ for all $x\geq 0$ and all $i=1,\ldots, N$, and since $\psi_i$ is strictly convex, we get, for all $i=1,\ldots, N$, that
$$
\psi_i'(\alpha_i) - \psi_i'(0)=\psi_i'(\alpha_i) \geq \psi'_i\left( \frac{\xi}{c_i}( e^{\lambda_0/(\gamma_i\xi)}  - 1) \right) = \gamma_i c_i \left( 1 - e^{-\lambda_0/(\gamma_i\xi)} \right).
$$
Recalling \cref{eq:up_bound_Fxstar}, a sufficient condition for both~\cref{eq:offsup} and~\cref{ass2genthm} to hold is given by
$$
\|\ba_i\|_2 \sqrt{2\tilde Lf(\|\beps\|_\infty, m(\bx^*))} < \min\{C_K, 1\}\gamma_i c_i\left(1 - e^{-\lambda_0/(\gamma_i\xi)} \right), \quad \text{for all } i\in{\sigma^*}^c.
$$
Finally, isolating $\lambda_0$ leads us to
$$
\lambda_0 >\max_{i\in{\sigma^*}^c}\left\{- \gamma_i \xi \log \left(1  - \frac{\|\ba_i\|_2 \sqrt{2\tilde Lf(\|\beps\|_\infty, m(\bx^*))} }{\min\{C_K, 1\} \gamma_i c_i} \right)\right\},
$$
concluding the first part of the result.

Next, we compute $\bar{\Lambda}$. To do so, we will provide sufficient conditions ensuring that $\mathfrak{S}\cap\bar\Omega=\emptyset$, where $\bar\Omega\in\{\Omega,\tilde\Omega\}$ depending on the value of $C_K$. To prove the result, we note that ensuring $\mathfrak{S}\cap\bar\Omega = \emptyset$ is equivalent to showing that every $\bu\in\bar\Omega$ is such that $\bu\notin \mathfrak{S}$. Notice that, by~\cref{propo:oracregKL}, if we provide conditions ensuring that
\begin{equation}\label{eq:necessKL}
\sum_{j=1}^{k^*}\gamma_{\sigma^*[j]}g^{\mathrm{KL}}_1\left(\frac{c_{\sigma^*[j]}u_j^*+\xi}{c_{\sigma^*[j]}u_j+\xi}\right)> \frac{F_{\sigma^*}(\bu^*)}{\xi C_K},
\end{equation}
then this implies that $\bu\notin\mathfrak{S}$, proving the result. To do so, and recalling again \cref{eq:up_bound_Fxstar}, it is enough to ensure that, for a given $j=j_0$, it holds true that (removing the subindices)
\begin{equation}\label{eq:fKL_lower_bound}
g^{\mathrm{KL}}_1\left(\frac{cu^*+\xi}{cu+\xi}\right)> \frac{f(\|\beps\|_\infty, m(\bx^*))}{\gamma \xi C_K}.
\end{equation}
Observe that, by definition, $g_1^{\mathrm{KL}}$ tends to $+\infty$ as long as the argument tends to $0$. Moreover, as it is coercive, it tends to $+\infty$ too whenever the argument tends to $+\infty$. Combining this, we see that its sublevel set at height $f(\|\beps\|_\infty, m(\bx^*))/(\gamma\xi C_K)$ is of the form $[e(\|\beps\|_\infty, m(\bx^*)), E(\|\beps\|_\infty, m(\bx^*))]$, for some $0<e(\|\beps\|_\infty, m(\bx^*))\leq E(\|\beps\|_\infty, m(\bx^*))<+\infty$. Then,~\cref{eq:fKL_lower_bound} is satisfied if
\begin{equation}\label{eq:geomKL}
\frac{cu^*+\xi}{cu+\xi}\geq E(\|\beps\|_\infty, m(\bx^*)).
\end{equation}
Now, we distinguish two cases: $C_K< 1$ and $C_K\geq1$. We start with the case $C_K< 1$, which in turn implies that $\bar\Omega=\tilde\Omega$. Let $\bu\in \tilde\Omega$ and observe that there exists $j_0\in\{1,\ldots, k^*\}$ such that $\psi'_{\sigma^*[j_0]}(u_{j_0})\leq \rho_{\sigma^*[j_0]}$. This is equivalent to (removing the subindices)
\begin{equation}\label{eq:uinOmegatildeKL}
\frac{1}{cu+\xi}\geq\frac{C_K(c\alpha+\xi)-c\alpha}{\xi C_K(c\alpha+\xi)}.
\end{equation}
Now, combining  \cref{eq:geomKL} with~\cref{eq:uinOmegatildeKL}, we get that~\cref{eq:fKL_lower_bound} is satisfied if
$$
\frac{cu^*+\xi}{\xi}\left(1-\frac{c\alpha}{C_K(c\alpha+\xi)}\right)\geq E(\|\beps\|_\infty, m(\bx^*)),
$$
or, equivalently, that
\begin{equation}\label{eq:suffcondupperboundKL}
1-\frac{c\alpha}{C_K(c\alpha+\xi)}\geq E'(\|\beps\|_\infty, m(\bx^*)):= \frac{\xi E(\|\beps\|_\infty, m(\bx^*))}{cu^*+\xi}
\end{equation}
Now, we consider the left hand side above as a function of $\alpha$, 
$$
g_1(\alpha):=\left(1-\frac{c\alpha}{C_K(c\alpha+\xi)}\right),
$$
and notice that it is strictly decreasing on its domain with $\lim_{\alpha \to \infty} g_1(\alpha) = 1 - 1/C_K$. Hence, it admits a well defined inverse on $(1-1/C_K,\infty)$. As $C_K< 1$, we have that $1-1/C_K<0$, and so, since $E'(\|\beps\|_\infty, m(\bx^*))>0$ for all $\|\beps\|_\infty$ and all $m(\bx^*)$, the inverse $g_1^{-1}$ is always well defined. Therefore,~\cref{eq:suffcondupperboundKL} is equivalent to $\alpha\leq g_1^{-1}(\min\{1, E'(\|\beps\|_\infty, m(\bx^*))\})$ and so, plugging the value of $\alpha$, we get that
$$
\frac{-\xi}{c}\left(\frac{1}{W(-e^{-\kappa})}+1\right)\leq g_1^{-1}(\min\{1, E'(\|\beps\|_\infty, m(\bx^*))\}).
$$
Isolating the Lambert function above leads to
\begin{equation}\label{eq:lambertupperbound}
W(-e^{-\kappa})\leq -\frac{\xi}{cg_1^{-1}(\min\{1, E'(\|\beps\|_\infty, m(\bx^*))\})+\xi}.
\end{equation}
Let us define define the function (where we omit the dependence on $\xi$ and $c$ for simplicity)
\begin{equation}\label{eq:h1interpret}
h^1(\|\beps\|_\infty, m(\bx^*)):=-\frac{\xi}{cg_1^{-1}(\min\{1, E'(\|\beps\|_\infty, m(\bx^*))\})+\xi}.
\end{equation} 
Next, notice that the following result holds, see \cite[Theorem 2.3]{Hoorfar2008}): if $x>-1/e$ and $y>1/e$, then $W(x)\leq \log[(x+y)/(1+\log(y))]$. Setting $y=1>1/e$ leads to
\begin{equation}\label{eq:Lambertineq}
    W(x)\leq \log(1+x), \quad \text{for all } x>-1/e.
\end{equation}
Combining~\cref{eq:lambertupperbound} with~\cref{eq:h1interpret} and~\cref{eq:Lambertineq}, it suffices now to show that
$$
\log(1-e^{-\kappa})\leq h^1(\|\beps\|_\infty, m(\bx^*)).
$$
Notice that the above inequality is well-posed: first, $h^1$ is negative by definition. On the other hand, the left hand side $\log(1-e^{-\kappa})$ is negative too since, as $\kappa>1$, we have $e^{-\kappa}<1$ and $\log(1 - x ) <0$ for all $x \in [0,1)$. As the exponential function is increasing, the above is equivalent to
$$
e^{-\kappa}\geq 1-e^{h^1(\|\beps\|_\infty, m(\bx^*))}
$$
or, in other words, that
$$
\kappa=\frac{\lambda_0}{\gamma\xi}+1\leq -\log\left(1-e^{ h^1(\|\beps\|_\infty, m(\bx^*))}\right)
$$
Recovering the dependence of $h^1$ on $i$ through the $c_i's$ and $\gamma_i$'s, and isolating $\lambda_0$ above, we have that it suffices that
$$
\lambda_0\leq \min_{i\in\sigma^*}\left\{-\gamma_i\xi\left(1+\log\left(1-e^{ h_i^1(\|\beps\|_\infty, m(\bx^*))}\right)\right)\right\}.
$$
By taking $h_i:=h_i^1$ for all $i\in\sigma^*$, we conclude that~\cref{eq:upperlambKL} provides a sufficient condition ensuring that \cref{eq:necessKL} holds. Hence, $\bu\notin\mathfrak{S}$ as desired. To finish this part, we look at the variations of and limits of $h_i^1$ for all $i\in\sigma^*$. Removing the dependence on $i$, we first get from the variations and limits of $f$ that $E(\|\beps\|_\infty, m(\bx^*))$ decreases toward $1$ whenever $\|\beps\|_\infty \to 0$ or $m(\bx^*) \to \infty$ (given that $g_1^{\mathrm{KL}}(1)=0$). Then, $E'(\|\beps\|_\infty, m(\bx^*))$ decreases toward $\xi / (cu^* +\xi)$ (resp. $0$) when $\|\beps\|_\infty \to 0$ (resp. $m(\bx^*) \to \infty$), and so
$g_1^{-1}(\min\{1,E'(\|\beps\|_\infty, m(\bx^*))\})$ increases towards
\begin{itemize}
    \item[$\triangleright$]  $g_1^{-1}\left(\xi /(cu^* + \xi)\right) = C_K \xi u^* /(c u^*(1-C_K) + \xi)$ when $\|\beps\|_\infty \to 0$,
    \item[$\triangleright$]  $g_1^{-1}(0) = C_K \xi /(c(1-C_K))$ when $m(\bx^*) \to \infty$.
\end{itemize}
Hence, $h^1(\|\beps\|_\infty, m(\bx^*))$ increases toward
  \begin{itemize}
    \item[$\triangleright$] $(c u^*(C_K-1)-\xi)/(cu^* + \xi)$ when $\|\beps\|_\infty \to 0$,
    \item[$\triangleright$] $C_K -1$ when $m(\bx^*) \to \infty$.
\end{itemize}

To complete the proof, we deal with the case $C_K\geq 1$, where $\bar\Omega=\Omega$. Let $\bu\in\Omega$. Then, there exists $j_0\in\{1,\ldots, k^*\}$ such that $0\leq u_{j_0}\leq \alpha_{\sigma^*[j_0]}$. This condition is equivalent to (removing the subindices),
$$
cu+\xi\leq c\alpha+\xi.
$$
As we did in the case above, if we combine the latter bound with \cref{eq:geomKL}, it suffices that
$$
\frac{cu^*+\xi}{c\alpha+\xi}\geq E(\|\beps\|_\infty,m(\bx^*)),
$$
or, equivalently, that
$$
g_2(\alpha):=\frac{1}{c\alpha+\xi}\geq E''(\|\beps\|_\infty,m(\bx^*)):=\frac{E(\|\beps\|_\infty,m(\bx^*))}{cu^*+\xi}
$$
where $g_2$ is strictly decreasing for all $\alpha\geq 0$  with $g_2(0) = 1/\xi$, and $\lim_{\alpha \to \infty} g_2(\alpha) = 0$. Hence, it admits a well defined inverse on $(0, 1/\xi)$. Then, the above is equivalent to 
$\alpha \leq g_2^{-1}\left(\min\{1/\xi, E''(\|\beps\|_\infty,m(\bx^*)) \}\right)$. Analogously to the previous case, we define
\begin{equation}\label{eq:h2interpret}
h^2(\|\beps\|_\infty,m(\bx^*)):=-\frac{\xi}{cg_2^{-1}\left(\min\{1/\xi, E''(\|\beps\|_\infty,m(\bx^*))\}\right)+\xi},
\end{equation}
and proceed in the same way to isolate $\lambda_0$.
Recovering the dependence of $h^2$ on $i$ through the $c_i$'s and $\gamma_i$'s, and taking $h_i:=h_i^2$ for all $i\in\sigma^*$, leads to the desired result.
To conclude, we look at the variations of and limits of $h_i^2$ for all $i\in\sigma^*$. Removing the dependence on $i$ and recalling the behaviour of $E(\|\beps\|_\infty, m(\bx^*))$, we have that $E''(\|\beps\|_\infty, m(\bx^*))$ decreases towards $(cu^* +\xi)^{-1}$ (resp. $0$) when $\|\beps\|_\infty \to 0$ (resp. $m(\bx^*) \to \infty$). Then, $g_2^{-1}(\min\{1/\xi,E''(\|\beps\|_\infty, m(\bx^*))\})$ increases towards
\begin{itemize}
    \item[$\triangleright$] $g_2^{-1}\left(1/(cu^* + \xi)\right) =u^*$ when $\|\beps\|_\infty \to 0$,
    \item[$\triangleright$] $+\infty$ when $m(\bx^*) \to \infty$.
\end{itemize}
Finally, $h^1(\|\beps\|_\infty, m(\bx^*))$ increases toward
\begin{itemize}
    \item[$\triangleright$] $-\xi/(cu^* + \xi)$ when $\|\beps\|_\infty \to 0$,
    \item[$\triangleright$] $0$ when $m(\bx^*) \to \infty$.
\end{itemize}
The limits and variations derived in~\cref{coro:KL_application} follow by combining the variations and limits for $h^1$ and $h^2$.

\section{Acknowledgements}
This work has been supported by the ANR EROSION (ANR-22-CE48-0004) and the Toulouse AI cluster ANITI (ANR-23-IACL-0002).

\bibliographystyle{siamplain}
\bibliography{samplebib}

\end{document}